\documentclass[12pt,a4paper]{article}
\usepackage{amssymb,amsmath,amsthm}
\usepackage{mathtools}
\usepackage{mathrsfs}
\usepackage{graphicx}
\usepackage{xtab}
\usepackage{xcolor}
\usepackage{bbm}
\usepackage{enumitem}
\usepackage{url}
\usepackage{bbm}
\usepackage{tikz}
\usetikzlibrary{matrix,arrows}

 \newcommand{\Z}{\ensuremath{\mathbb{Z}}}% integers
 \newcommand{\N}{\ensuremath{\mathbb{N}}}% natural numbers
 \newcommand{\R}{\ensuremath{\mathbb{R}}}% real numbers
\setlist[enumerate,1]{label=(\roman*)}

\newcommand{\norm}[1]{\left\lVert#1\right\rVert}
\newcommand{\pnorm}[1]{\left|#1\right|}
\newcommand{\norminf}[1]{\left|#1\right|_\infty}
\newcommand{\eqd}{=_d}
\newcommand{\holsp}[1][\eta]{\mathcal{C}^{#1}}

\newcommand{\Ex}[2][]{\mathbb{E}_{#1\!\!}\left[#2\right]}

\newcommand{\indic}[1]{\mathbbm{1}\left\{#1\right\}}
\newcommand{\indicd}[1]{\mathbbm{1}_{#1}}

\newcommand{\bbX}{\ensuremath{\mathbb{X}}}% integers
\newcommand{\bbW}{\ensuremath{\mathbb{W}}}% integers
\newcommand{\bbS}{\ensuremath{\mathbb{S}}}% integers

\newcommand{\tf}{T}
\newcommand{\ms}{M}
\newcommand{\conv}{\rightarrow}

\newcommand{\eps}{\varepsilon}

\newcommand{\wconv}{\ensuremath{\overset{w}{\longrightarrow}}}

\newcommand{\leb}{\mathrm{Leb}}
\newcommand{\holder}{H\"{o}lder}

\newcommand{\ito}{It\^{o}}
\newcommand{\cadlag}{c{\`a}dl{\`a}g}

\newcommand{\map}[2]{\colon#1\rightarrow#2}
\newcommand{\Lip}{\mathrm{Lip}}
\usepackage[nodayofweek]{datetime} % change the format of printed dates (no american style!) 
\newdateformat{monthyear}{\monthname[\THEMONTH], \THEYEAR}% new date format

\newtheorem{prop}{Proposition}[section]
\newtheorem{lemma}[prop]{Lemma}

\newtheorem{theorem}[prop]{Theorem}

\newtheorem{defn}[prop]{Definition}
\newtheorem{exa}[prop]{Example}

\theoremstyle{remark}
\newtheorem{remark}[prop]{Remark}

\makeatletter
\let\save@mathaccent\mathaccent
\newcommand*\if@single[3]{%
	\setbox0\hbox{${\mathaccent"0362{#1}}^H$}%
	\setbox2\hbox{${\mathaccent"0362{\kern0pt#1}}^H$}%
	\ifdim\ht0=\ht2 #3\else #2\fi
}
\newcommand*\rel@kern[1]{\kern#1\dimexpr\macc@kerna}
\newcommand*\widebar[1]{\@ifnextchar^{{\wide@bar{#1}{0}}}{\wide@bar{#1}{1}}}
\newcommand*\wide@bar[2]{\if@single{#1}{\wide@bar@{#1}{#2}{1}}{\wide@bar@{#1}{#2}{2}}}
\newcommand*\wide@bar@[3]{%
	\begingroup
	\def\mathaccent##1##2{%
		\let\mathaccent\save@mathaccent
		\if#32 \let\macc@nucleus\first@char \fi
		\setbox\z@\hbox{$\macc@style{\macc@nucleus}_{}$}%
		\setbox\tw@\hbox{$\macc@style{\macc@nucleus}{}_{}$}%
		\dimen@\wd\tw@
		\advance\dimen@-\wd\z@
		\divide\dimen@ 3
		\@tempdima\wd\tw@
		\advance\@tempdima-\scriptspace
		\divide\@tempdima 10
		\advance\dimen@-\@tempdima
		\ifdim\dimen@>\z@ \dimen@0pt\fi
		\rel@kern{0.6}\kern-\dimen@
		\if#31
		\overline{\rel@kern{-0.6}\kern\dimen@\macc@nucleus\rel@kern{0.4}\kern\dimen@}%
		\advance\dimen@0.4\dimexpr\macc@kerna
		\let\final@kern#2%
		\ifdim\dimen@<\z@ \let\final@kern1\fi
		\if\final@kern1 \kern-\dimen@\fi
		\else
		\overline{\rel@kern{-0.6}\kern\dimen@#1}%
		\fi
	}%
	\macc@depth\@ne
	\let\math@bgroup\@empty \let\math@egroup\macc@set@skewchar
	\mathsurround\z@ \frozen@everymath{\mathgroup\macc@group\relax}%
	\macc@set@skewchar\relax
	\let\mathaccentV\macc@nested@a
	\if#31
	\macc@nested@a\relax111{#1}%
	\else
	\def\gobble@till@marker##1\endmarker{}%
	\futurelet\first@char\gobble@till@marker#1\endmarker
	\ifcat\noexpand\first@char A\else
	\def\first@char{}%
	\fi
	\macc@nested@a\relax111{\first@char}%
	\fi
	\endgroup
}
\makeatother

\title{Functional correlation bounds and deterministic homogenisation of fast-slow systems}
\author{Nicholas Fleming-V\'{a}zquez\thanks{Mathematics Institute, University of Warwick, Coventry, CV4 7AL, UK.}}
\date{}
\usepackage{hyperref}
\usepackage{fullpage}
\hypersetup{
	colorlinks=true,
	linkcolor=blue,
	filecolor=magenta,      
	urlcolor=cyan,
}

\numberwithin{equation}{section}

\newcommand{\shat}[1]{\vphantom{#1}\smash[t]{\widehat{#1}}}
\newcommand{\stilde}[1]{\vphantom{#1}\smash[t]{\widetilde{#1}}}
\newcommand{\fcb}{Functional Correlation Bound}

\newcommand{\sepdholsp}[1]{\mathcal{H}^\eta_{#1}}
\newcommand{\dsemi}[1]{[#1]_{\eta}}
\newcommand{\sdsemi}[2]{[#1]_{\eta,#2}}
\newcommand{\dhnorm}[1]{\norm{#1}_\eta}

\begin{document}
	\maketitle
		\begin{abstract}
		We give elementary and explicit sufficient conditions (in particular, a functional correlation bound) for deterministic homogenisation (convergence to a stochastic differential equation) for discrete-time fast-slow systems of the form 
		\[ x_{k+1}=x_k+n^{-1}a_n(x_k,y_k)+n^{-1/2}b_n(x_k,y_k), \quad y_{k+1}=T_n y_k. \]
		We then prove that these sufficient conditions are satisfied by various examples of nonuniformly hyperbolic dynamical systems $ T_n $.
	\end{abstract}
\section{Introduction}
Let $ \ms $ be a metric space and let $ T:M\to M $ be a transformation preserving an ergodic Borel probability measure $ \mu $. Consider a fast-slow system on $ \R^d \times M $ of the form 
\begin{equation}\label{eq:fast_slow_eqn_single}
	x^{(n)}_{k+1}=x^{(n)}_k+n^{-1}a(x^{(n)}_k,y_k)+n^{-1/2}b(x^{(n)}_k,y_k), \qquad y_{k+1}=T y_k
\end{equation}
where $ x^{(n)}_0 \equiv \xi $ is fixed and $ y_0 $ is drawn randomly from $ (M,\mu) $. Assume that $ \int_M b(x,y)d\mu(y)=0 $ for all $ x\in \R^d $. If the fast dynamics $ T:M\to M $ is chaotic enough, then we expect the stochastic process $ X_n(t)=x^{(n)}_{[nt]} $ to converge weakly to the solution of a stochastic differential equation driven by Brownian motion. This is referred to as \textit{(deterministic) homogenisation} and has attracted a great deal of interest recently \cite{dolgopyat2004limit,kelly2016smooth,kelly2017dethom,korepanov2022deterministic,chevyrev2022deterministic}.

Kelly \& Melbourne~\cite{kelly2016smooth} used rough path theory to show that homogenisation reduces to showing certain statistical properties for the fast dynamics $ T:M\to M $. The required statistical properties were recently relaxed in~\cite{chevyrev2022deterministic}. In~\cite{melbourne2016note,korepanov2022deterministic,fleming2022} these statistical properties were verified for a large class of nonuniformly hyperbolic maps.

Homogenisation results have many interesting physical applications (cf.\!~\cite[Sect.~11.8]{pavliotis_stuart}), particularly in stochastic climate theory~\cite{stochastic_climate_theory_17}. As such, it is desirable to find a sufficient condition for homogenisation that is accessible to a broad audience.

For many classes of chaotic dynamical systems (particularly ones with some hyperbolicity), it is possible to prove bounds on the correlations 
\[ C_n(v,w)=\int_M v\,w\circ T^n d\mu - \int_M v\, d\mu\int_M w \,d\mu   \]
for all \holder{} $ v$, $w:M\to \R $. Moreover, correlation bounds play a crucial role in most standard proofs of the central limit theorem, so it is natural to seek a sufficient condition for homogenisation in terms of decay of correlations.

Fast decay of the autocorrelations $ C_n(v,v) $ is not enough to guarantee that $ v:M\to \R $ satisfies the central limit theorem, see~\cite{griffin_marklof} for a counterexample. More generally,  fast decay of correlations for \holder{} observables is not thought to be sufficient to prove the central limit theorem.

Let $ v=(v_1,\dots,v_d)\in L^\infty(M,\R^d) $ with $ \int v\, d\mu =0 $. Suppose that there exists a sequence $ a_n>0 $ such that $ \sum_{n\ge 1}a_n^{1/3}<\infty $ and 
\begin{equation}\label{eq:gordin_criterion}
	\big|C_n(v_i,w)\big|\le a_n\pnorm{w}_\infty \text{ for all } w\in L^\infty(M), 1\le i\le d. 
\end{equation}
Then $ v $ satisfies the statistical properties required by \cite{kelly2016smooth}. However, if $ T $ is invertible, then this condition fails whenever $ v\ne 0 $. We now give a generalised correlation condition that suffices to prove homogenisation and can be satisfied by both invertible and noninvertible maps.

Fix $ \eta\in (0,1] $ and let $ v:M\to \R $. Let $ \dsemi{v}=\sup_{x\ne y}|v(y)-v(x)|/d(x,y)^\eta $ denote the $ \eta $-H\"{o}lder seminorm of $ v $. Let $ q\ge 1 $. For $ G\map{\ms^{q}}{\R} $ and $ 0\le i< q$ we denote
	\[\sdsemi{G}{i} =\sup_{x_0,\dots,x_{q-1}\in \ms}\dsemi{G(x_0,\dots,x_{i-1},\cdot,x_{i+1},\dots,x_{q-1})}. \]
	We call $ G $ \textit{separately $ \eta $-\holder{}}, and write $ G\in \sepdholsp{q}(\ms)$, if $ \norminf{G}+\sum_{i=0}^{q-1} \sdsemi{G}{i}<\infty. $

Note that $\sepdholsp{1}(\ms)=\holsp(\ms) $ is the space of $ \eta $-\holder{} observables.

Let $\gamma>0$. We consider dynamical systems which satisfy the following type of correlation bound:
\begin{defn}
	Let $ \tf\map{\ms}{\ms} $ be a dynamical system with ergodic invariant probability measure $ \mu $. Suppose that there exists a constant $C>0$ such that for all integers $ 0\le p< q,\ 0\le k_0\le \cdots\le k_{q-1} $,
	\begin{align}
		&\bigg|\int_\ms  G(\tf^{k_0}x,\dots,\tf^{k_{q-1}}x)d\mu(x)\nonumber\\
		&\qquad\qquad\qquad -\int_{\ms^2} G(\tf^{k_0}x_0,\dots,\tf^{k_{p-1}}x_0,\tf^{k_p}x_1,\dots,\tf^{k_{q-1}} x_1)d\mu(x_0)d\mu(x_1)\bigg|	\nonumber\\
		&\qquad\le C(k_p-k_{p-1})^{-\gamma}\biggl(\norminf{G}+\sum_{i=0}^{q-1} \sdsemi{G}{i}\biggr)	\label{eq:fcb_bd}	
	\end{align}
	for all $G\in \sepdholsp{q}(\ms)$. Then we say that $\tf$ satisfies the Functional Correlation Bound with rate $k^{-\gamma}$.
\end{defn}
Conditions of this kind were introduced by Lepp\"{a}nen in~\cite{leppanen2017functional} and were further studied in~\cite{leppanen2017billiards,leppanen2020sunklodas}. In \cite{fleming2022}, the author showed that the \fcb{} is satisfied by nonuniformly hyperbolic maps modelled by Young towers with polynomial tails. In particular, this covers many well-known examples such as Axiom A attractors, H\'enon maps with Benedicks-Carleson parameters and intermittent interval maps. Our first main result is that homogenisation holds whenever the Functional Correlation Bound is satisfied with a fast enough rate. Let $ D([0,1],\R^d) $ denote the Skorohod space with uniform topology.
\begin{theorem}\label{thm:homogenisation_single_map}Suppose that $ T $ satisfies the Functional Correlation Bound with rate $ k^{-\gamma} $, $ \gamma>1 $. Let $ a,b:\R^d \times M\to \R^d $ be regular. Then $ X_n \wconv X$ in $ D([0,1],\R^d) $, where $ X $ is an It\^{o} diffusion.
\end{theorem}
We refer to~\cite[Sect.~2]{chevyrev2022deterministic} for a precise statement of the regularity requirements on $ a $ and $ b $, in addition to explicit expressions for the drift and diffusion coefficients of the stochastic differential equation satisfied by $ X $. 

The condition $ \gamma > 1$ in Theorem~\ref{thm:homogenisation_single_map} is sharp. For $ \gamma\le 1 $ there are examples where the central limit theorem fails for generic H\"{o}lder observables (see Example~\ref{exa:intermittent_baker} for more details).
\subsection{More general fast-slow systems}
Until now we have assumed that the fast direction $ y_k $ in \eqref{eq:fast_slow_eqn_single} is independent of $ n $. Let $ T_n:M\to M $ be a family of maps with invariant probability measures $ \mu_n $. We generalise~\eqref{eq:fast_slow_eqn_single} by considering fast-slow systems on $ \R^d\times M $ of the form
\begin{equation}\label{eq:fast_slow_family}
	x^{(n)}_{k+1}=x^{(n)}_k+n^{-1}a_n(x^{(n)}_k,y^{(n)}_k)+n^{-1/2}b_n(x^{(n)}_k,y^{(n)}_k), \quad y^{(n)}_{k+1}=T_n y^{(n)}_k.
\end{equation}
Here $ x_0^{(n)}\equiv \xi $ is fixed and $ y_0^{(n)} $ is drawn randomly from $ (M,\mu_n) $. The homogenisation problem we are interested in is the same as before, only now $ X_n $ is a stochastic process on $ (M,\mu_n) $. This homogenisation problem was previously considered by~\cite{kkm_martinale_cobdy,korepanov2022deterministic,chevyrev2022deterministic}. In particular, in~\cite{korepanov2022deterministic,chevyrev2022deterministic} this problem was settled for families of nonuniformly expanding maps.

%In \cite{kkm_martinale_cobdy} homogenisation was considered in the special case where $ b_n(x,y)=h_n(x)v_n(y) $ with $ h_n $ exact. Chevyrev et al.~\cite{chevyrev2022deterministic} then developed the rough-path theory necessary to consider general $ a_n, b_n $. The article~\cite{korepanov2022deterministic} then used this theory to prove homogenisation for families of nonuniformly expanding maps. 
% In~\cite{korepanov2022deterministic} the homogenisation problem is addressed in the case where $ T_n $ is a family of nonuniformly expanding maps. 

We now give a generalisation of Theorem~\ref{thm:homogenisation_single_map} for families $ T_n $ that satisfy the Functional Correlation Bound in the following uniform sense.

Let $ \tf_n\map{\ms}{\ms}, $ $ n\in \N\cup\{\infty\}, $ be a family of dynamical systems with invariant probability measures $ \mu_n $. Suppose that $ \tf_n $ satisfies the \fcb{} with rate $ k^{-\gamma} $ for each $ n $, and that the constant $ C>0 $ can be chosen independently of $ n $. Then we say that the family $ \tf_n $ \textit{satisfies the \fcb{} uniformly with rate $ k^{-\gamma} $}. Our main result on homogenisation for fast-slow systems of the form \eqref{eq:fast_slow_family} is as follows:
\begin{theorem}\label{thm:homog_from_fcb_family}
	Let $ T_n:M\to M $ be a family of dynamical systems that satisfies the Functional Correlation Bound uniformly with rate $ k^{-\gamma} $, $ \gamma>1 $. Suppose that
\begin{equation}\label{eq:conv_correlation_fns}
	\lim_{n\rightarrow \infty}\int_M vw\circ T_n^j\, d\mu_n=\int_M vw\circ T_\infty^j \,d\mu_\infty
\end{equation}
for all $ j\ge 0 $ and all $ \eta $-H\"{o}lder $ v,\, w:M\to \R $.
	Assume that $ a_n$, $b_n:\R^d\times M\rightarrow \R^d $ satisfy mild regularity conditions. Then $ X_n\wconv X$ in $ D([0,1],\R^d) $, where $ X $ is an It\^{o} diffusion.
\end{theorem}
Again, we refer to~\cite[Sect.~2]{chevyrev2022deterministic} for a precise statement of the regularity conditions on $ a_n $ and $ b_n $. Note that Theorem~\ref{thm:homogenisation_single_map} follows from Theorem~\ref{thm:homog_from_fcb_family} by taking $ T_n\equiv T $, $ \mu_n\equiv \mu $. 
\begin{remark}\label{rk:sufficient_conditions_conv}
	It is straightforward to check that condition~\eqref{eq:conv_correlation_fns} holds provided that 
		\begin{enumerate}[label=(A\arabic*),ref=(A\arabic*)]
		\item $ T_\infty $ is continuous $ \mu_\infty $-a.e.\ and $ \mu_n $ is statistically stable, i.e.\ $ \mu_n \wconv \mu_\infty $.
		\item $\lim_{n\rightarrow \infty}\mu_n(x\in M: d(T_n^j x, T_\infty^j x)>a)=0$ for all $ a>0 $ and $ j\ge 0 $.
	\end{enumerate}
\end{remark}
Conditions (A1)-(A2) are similar to condition (7.2) in~\cite{kkm_martinale_cobdy}. Due to strong statistical stability, these conditions are easy to verify for families of nonuniformly expanding maps. However, condition~\eqref{eq:conv_correlation_fns} is easier to check in certain situations. In particular, in Subsection~\ref{subsection:sinai_billiard_external} we show that it is also possible to verify condition~\eqref{eq:conv_correlation_fns} for an example where statistical stability is proved by Keller-Liverani perturbation theory.

We end this introduction by giving a simple example of a family of dynamical systems covered by Theorem~\ref{thm:homog_from_fcb_family}.
\begin{exa}\label{exa:intermittent_baker}
	Let $ \alpha\in (0,1) $. Define $ g_\alpha:[0,\frac{1}{2}]\to [0,1] $ by $ g_\alpha(x)=x(1+2^\alpha x^\alpha) $. In~\cite{lsvmaps} Liverani, Saussol \textsl{\&} Vaienti introduced an intermittent map of the form
\begin{equation*}
	\bar T_\alpha:[0,1]\to [0,1], \quad \bar T_\alpha(x)=\begin{cases}
		g_\alpha(x), & x< \frac{1}{2},\\
		2x - 1, & x\ge\frac{1}{2}.
	\end{cases}
\end{equation*}
The map $ \bar T_\alpha $ is an archetypal example of a nonuniformly expanding map with slow decay of correlations. Let $ M=[0,1]^2 $. Consider an intermittent Baker's map~\cite{melbourne2016note} given by
\begin{equation}\label{eq:intermittent_baker}
	T_\alpha:M\to M, \quad T_\alpha(x,y)=\begin{cases}
		(\bar T_\alpha(x), g_\alpha^{-1}(y)), & x< \frac{1}{2},\\
		(\bar T_\alpha(x), \frac{1}{2}(y+1)), & x \ge \frac{1}{2}.
	\end{cases}
\end{equation}
The map $ T_\alpha $ is invertible and has a unique physical measure $ \mu_\alpha $. By~\cite{fleming2022}, $ T_\alpha $ satisfies the Functional Correlation Bound with rate $ k^{1-1/\alpha} $. For $ \alpha\ge \frac{1}{2} $, the central limit theorem fails for generic H\"older observables even for the $ \bar T_\alpha $ dynamics~\cite{gouezel2004central}. Hence it is natural to restrict to the range $ \alpha<\frac{1}{2} $ when considering homogenisation.

Let $ \alpha_n\in (0,\frac{1}{2}) $ satisfy $ \alpha_n\to \alpha_\infty\in (0,\frac{1}{2})  $. Let $ \bar T_n = \bar T_{\alpha_n} $, $ T_n = T_{\alpha_n} $ and $ \mu_n = \mu_{\alpha_n} $. In \cite{korepanov2022deterministic} homogenisation is obtained for the family $ \bar T_n $. By Theorem~\ref{thm:homog_from_fcb_family}, we obtain homogenisation for the family $ T_n $; see Subsection~\ref{subsection:intermittent_baker} for more details.
\end{exa}

The remainder of this article is structured as follows. In Section~\ref{section:pf_main_result}, we prove Theorem~\ref{thm:homog_from_fcb_family}. In Section~\ref{section:families_nuh_maps}, we consider nonuniformly hyperbolic maps in the sense of Young and introduce regularity conditions on families of such maps. In particular, in Theorem~\ref{thm:uniform_fcb} we show that such families satisfy the Functional Correlation Bound with uniform rate. Finally, in Section~\ref{section:examples}, we demonstrate how our results apply to externally forced dispersing billiards and the family of intermittent Baker's maps from Example~\ref{exa:intermittent_baker}.

\quad
\textbf{Notation: }We write $ \rightarrow_{\mu_n} $ to denote weak convergence with respect to a specific family of probability measures $ \mu_n $ on the left-hand side. So $ X_n\rightarrow_{\mu_n} X $ means that $ X_n $ is a family of random variables on $ (\ms, \mu_n) $ and $ X_n \wconv X $.

We write $a_n=O(b_n) $ or $ a_n\ll b_n $ if there exists a constant $ C>0 $ such that $ a_n\le Cb_n $ for all $ n\ge 1. $ We write $ a_n=o(b_n) $ if $ \lim_n a_n/b_n=0 $ and $ a_n\sim b_n $ if $ \lim_n a_n/b_n=1. $

For $ d\ge 1 $ and $ a,b\in \R^d $ denote $ a\otimes b=ab^T. $ We endow $ \R^d $ with the norm $ |y|=\sum_{i=1}^d |y_i|. $

For real-valued functions $ f $ and $ g $, we let $\int f \, dg $ denote the \ito{} integral (where defined). For $ \R^d $-valued functions, $ \int f\otimes dg $ denotes the matrix of \ito{} integrals.

For $ \eta\in (0,1] $ and $ v\map{\ms}{\R} $ we denote the $ \eta $-\holder{} norm by $ \dhnorm{v}=\norminf{v}+\dsemi{v} $. If $ \eta=1 $ we call $ v $ Lipschitz and write $ \Lip(v)=[v]_1 $.

\quad\textbf{Acknowledgements:} The author is supported by the Warwick Mathematics Institute Centre for Doctoral Training, and gratefully acknowledges funding from the University of Warwick. The author would also like to thank his supervisor Ian Melbourne for the extensive support and feedback that he provided whilst this article was being written. Finally, he would like to thank Alexey Korepanov for helpful discussions and comments.
\section{Proof of the main result}\label{section:pf_main_result}
This section is dedicated to the proof of Theorem~\ref{thm:homog_from_fcb_family}. We proceed by applying~\cite[Theorem~2.17]{chevyrev2022deterministic}. In Subsection~\ref{subsection:iwip_preliminaries}, we recall some consequences of the Functional Correlation Bound from \cite{fleming2022}, including iterated moment bounds. In Subsection~\ref{subsection:iterated_wip}, we prove the iterated WIP, which is the other main hypothesis of~\cite[Theorem~2.17]{chevyrev2022deterministic}. Finally, in Subsection~\ref{subsection:pf_main_result}, we prove Theorem~\ref{thm:homog_from_fcb_family}.
\subsection{Preliminaries}\label{subsection:iwip_preliminaries}
Let $ T:M\to M $ be a map that satisfies the Functional Correlation Bound with rate $ k^{-\gamma} $, $ \gamma > 0 $.
We repeatedly use the following weak dependence lemma.

Let $ e,q\ge 1 $ be integers. For $ G=(G_1,\dots,G_e)\map{\ms^q}{\R^e}$ and $ 0\le i<q $ we define $ \sdsemi{G}{i}=\sum_{j=1}^e \sdsemi{G_j}{i}. $ We write $ G\in \sepdholsp{q}(\ms,\R^e) $ if $ \norminf{G}+\sum_{i=0}^{q-1}\sdsemi{G}{i}<\infty. $
Let $k\ge 1$ and consider $ k $ disjoint blocks of integers $\{\ell_i,\ell_i+1,\dots,u_i\}$, $1\le i\le k$ with $ \ell_i\le u_i<\ell_{i+1}. $ Consider $ \R^e $-valued random vectors $ X_i $ on $ (\ms,\mu) $ of the form
\begin{equation*}
	X_i(x)=\varPhi_i(\tf^{\ell_i}x,\dots,\tf^{u_i}x)
\end{equation*}
where $ \varPhi_i\in \sepdholsp{u_i-\ell_i+1}(\ms,\R^e) $, $ 0\le i< k. $

When the gaps $\ell_{i+1}-u_i$ between blocks are large, the random vectors $ X_1,\dots,X_k $ are weakly dependent. Let $\widehat{X}_1,\dots,\widehat{X}_k$ be independent random vectors with $\widehat{X}_i{\eqd X_i}$.
\begin{lemma}\label{lemma:fcb_weak_dep_multidim}
Let $R=\max_i \norminf{\varPhi_i}$. Then for all Lipschitz $ F\map{B(0,R)^k}{\R} $,
	\begin{multline*}
		\big|\Ex[\mu]{F(X_1,\dots,X_k)}-\Ex{F(\shat{X}_1,\dots,\shat{X}_k)}\big|\\
		\le C\sum_{r=1}^{k-1}(\ell_{r+1}-u_r)^{-\gamma}\biggl(\norminf{F}+\Lip(F)\sum_{i=1}^{k}\sum_{j=0}^{u_i-\ell_i}\sdsemi{\varPhi_i}{j}\biggr),
	\end{multline*}
	where $ C>0 $ is the constant given by \eqref{eq:fcb_bd}.
\end{lemma}
\begin{proof}
	This is proved for $ e=1 $ in \cite[Lemma~4.1]{fleming2022}. The proof of the general case follows the same lines.
\end{proof}
\begin{lemma}\label{lemma:moment_bds}
	Let $ \gamma > 1 $. Then there exists a constant $ C>0 $ such that for all $ k\ge 1 $, for any mean zero $ v,\, w\in \holsp(\ms) $, 
	\begin{enumerate}
		\item $ \pnorm{\sum_{0\le r<k}v\circ \tf^r}_{2\gamma}\le Ck^{1/2}\dhnorm{v}. $
		\item $ \pnorm{\sum_{0\le r<s<k}v\circ \tf^r\ w\circ \tf^s}_{\gamma}\le Ck\dhnorm{v}\dhnorm{w}. $
	\end{enumerate}
\end{lemma}
\begin{proof}
	By \cite[Theorem~2.4]{fleming2022}, a functional correlation bound for separately \textit{dynamically} \holder{} functions implies analogous moment bounds for dynamically \holder{} observables. It is easily checked that the same arguments apply with separately \holder{} functions in place of separately dynamically \holder{} ones.
\end{proof}
\subsection{The Iterated WIP}\label{subsection:iterated_wip}
Let $ \gamma>1 $. Throughout this subsection, $ \tf_n,\, n\ge 1 $ is a family of maps that satisfies the \fcb{} uniformly with rate $ k^{-\gamma} $.

\sloppy Fix $ d\ge 1. $ Let $ v_n\map{\ms}{\R^d} $, $ n\ge 1 $ be a family of observables with $ \sup_{n\ge1}\dhnorm{v_n}<\infty $ and $ \int v_n\, d\mu_n=0 $. For $ t\ge 0 $ define
\[ W_n(t)=n^{-1/2}\sum_{0\le r<[nt]}v_n\circ \tf_n^r, \qquad \bbW_n(t)=n^{-1}\sum_{0\le r<s<[nt]}v_n\circ \tf_n^r\otimes v_n\circ \tf_n^s. \]
Let $ v\map{\ms}{\R^d} $ and $ k,\,n\ge 1 $. Define $ S_v(k,n)=\sum_{0\le r<k}v\circ \tf_n^r $ and $ \bbS_v(k,n)=\sum_{0\le r<s<k}v\circ \tf_n^r \otimes v\circ \tf_n^s. $
\begin{prop}\label{prop:uniform_conv_var}
	For each $ n\ge 1 $, the limits 
	\[\Sigma_n=\lim_{k\rightarrow\infty}k^{-1}\Ex[\mu_n]{S_{v_n}(k,n)\otimes S_{v_n}(k,n)}, \quad E_n=\lim_{k\rightarrow\infty}k^{-1}\Ex[\mu_n]{\bbS_{v_n}(k,n)}  \] 
	exist and are given by 
	\begin{align}\label{eq:covar_drift_formulas}
		\begin{split}
		\Sigma_n &= \Ex[\mu_n]{v_n\otimes v_n}+\sum_{\ell\ge 1}\big(\Ex[\mu_n]{v_n\otimes v_n\circ \tf_n^\ell}+\Ex[\mu_n]{v_n\circ \tf_n^\ell\otimes v_n}\! \big),\\
		E_n &= \sum_{\ell\ge 1}\Ex[\mu_n]{v_n\otimes v_n\circ \tf_n^\ell}.
		\end{split}
	\end{align}
	Moreover, the convergence is uniform in $ n $.
\end{prop}
\begin{proof}
	We prove the existence of the limit $E_n $. The proof of the existence of the limit $ \Sigma_n $ is similar. Note that
	\begin{align*}
		\Ex[\mu_n]{\bbS_{v_n}(k,n)}=\sum_{\ell=1}^{k-1}\sum_{r=0}^{k-\ell-1}\Ex[\mu_n]{v_n\otimes v_n\circ \tf_n^\ell}=\sum_{\ell=1}^{k-1}(k-\ell)\Ex[\mu_n]{v_n\otimes v_n\circ \tf_n^\ell}.
	\end{align*}
	Let $ 1\le i,j\le d $ and $ \ell\ge 1 $. Define $ G\map{\ms^2}{\R} $ by $ G(x,y)=v_n^i(x)v_n^j(y) $. By the \fcb{},
	\begin{align*}
		|\Ex[\mu_n]{v_n^i v_n^j \circ \tf_n^\ell}|&=\bigg|\int G(x,\tf_n^\ell x)d\mu_n(x)\bigg|\\
		&\ll \ell^{-\gamma}\dhnorm{v_n^i}\dhnorm{v_n^j}+\bigg|\int v_n^i d\mu_n \int v_n^j d\mu_n\bigg|=\ell^{-\gamma}\dhnorm{v_n^i}\dhnorm{v_n^j}.
	\end{align*}
	It follows that for all $ n\ge 1, $
	\begin{flalign*}
		\quad &\bigg|\sum_{\ell\ge 1}\Ex[\mu_n]{v_n\otimes v_n\circ \tf_n^\ell}-k^{-1}\Ex[\mu_n]{\bbS_{v_n}(k,n)}\!\bigg| &\\
		&\qquad\qquad\le k^{-1}\sum_{\ell=1}^{k-1}\ell\big|\Ex[\mu_n]{v_n\otimes v_n\circ \tf_n^\ell}\!\big|+\sum_{\ell\ge k}\big|\Ex[\mu_n]{v_n\otimes v_n\circ \tf_n^\ell}\!\big|\\
		& 
		\qquad\qquad\ll k^{-1}\sum_{\ell=1}^{k-1}\ell^{1-\gamma}+\sum_{\ell\ge k}\ell^{-\gamma}\ll k^{-1}(1+k^{2-\gamma})+k^{1-\gamma}=o(1),
	\end{flalign*}
	as required.
\end{proof}
We are now ready to state the main result of this subsection:
\begin{theorem}[Iterated WIP]\label{thm:iterated_wip}
	Suppose that $ \lim_{n\rightarrow\infty}\Sigma_n=\Sigma $ and $ \lim_{n\rightarrow\infty}E_n=E $. Then $ (W_n,\bbW_n)\rightarrow (W,\bbW) $ in the sense of finite-dimensional distributions, where $ W $ is a Brownian motion with covariance $ \Sigma $ and $ \bbW(t)=\int_0^t W\otimes dW+Et $. This means that for all $ \ell \ge 1 $ and $ 0\le t_1,\dots,t_\ell\le 1 $,
	\[ ((W_n,\bbW_n)(t_1),\dots,(W_n,\bbW_n)(t_\ell))\rightarrow_{\mu_n}((W,\bbW)(t_1),\dots,(W,\bbW)(t_\ell)). \]
\end{theorem}
Our proof of the Iterated WIP (Theorem~\ref{thm:iterated_wip}) is inspired by the proof of the central limit theorem in \cite[Chap.\ 7]{chernov_markarian}, which is based on Bernstein's `big block-small block' technique. Let $ 0<b<a<1 $. We split $ \{0,\dots,n-1\} $ into alternating big blocks of length $ p=[n^a] $ and small blocks of length $ q=[n^b] $. Let $ k $ denote the number of big blocks, which is equal to the number of small blocks. Then $ k=[n/(p+q)]=O(n^{1-a}). $ The last remaining block is of length at most $ p+q $.

Let $ \mathcal{B}\subset \{0,\dots,n-1\} $ denote the set of terms contained in big blocks. Let $ t\in [0,1] $. Then $ W_n(t)=I_1(t)+I_2(t) $ and $ \bbW_n(t)=J_1(t)+J_2(t)+J_3(t) $, where
\begin{align}
	\begin{split}\label{eq:iterated_big_block_cont}
		I_1(t)&=\frac{1}{n^{1/2}}\sum_{0\le r<[nt]\colon r\in \mathcal{B}}v_n\circ\tf_n^r,\qquad I_2(t)=\frac{1}{n^{1/2}}\sum_{0\le r<[nt] \colon r\notin \mathcal{B}}v_n\circ\tf_n^r,  \\
		J_1(t)&=\frac{1}{n}\sum_{0\le r<s<[nt]\colon r,s\in \mathcal{B}}v_n\circ \tf_n^r \otimes v_n\circ \tf_n^s,\\
		J_2(t)&=\frac{1}{n}\sum_{0\le r<s<[nt]\colon r\notin \mathcal{B},s\in \mathcal{B}}v_n\circ \tf_n^r \otimes v_n\circ \tf_n^s, \\ J_3(t)&=\frac{1}{n}\sum_{0\le r<s<[nt]\colon s\notin \mathcal{B}}v_n\circ \tf_n^r \otimes v_n\circ \tf_n^s.
	\end{split}
\end{align}
\begin{remark}
	In \cite[Chap.\ 7]{chernov_markarian} the central limit theorem is proved under a hypothesis on decay of multiple correlations, where the Functional Correlation Bound~\eqref{eq:fcb_bd} is only assumed for functions $ G:M^q\to \R $ of the form $ G(x_0,\dots,x_{q-1})=\prod_{i=0}^{q-1}v_i(x_i) $. This hypothesis is strong enough to control the characteristic function of $ I_1(t) $. However, functions $ G$ which are not of the above form arise naturally when we consider the characteristic function of $ J_1(t) $.
\end{remark}
We first show that the terms $ I_2(t) $, $ J_2(t) $, $ J_3(t) $ that involve small blocks can be neglected.
\begin{lemma}\label{lemma:small_blocks_negligible}
	Suppose that $ a>\frac{b+1}{2} $. Let $ t\in [0,1] $. Then $I_2(t)\rightarrow_{\mu_n}0$, $ J_2(t)\rightarrow_{\mu_n}0 $ and $ J_3(t)\rightarrow_{\mu_n}0 $ as $ n\rightarrow\infty $.
\end{lemma}
\begin{proof}
	We show that $ \pnorm{J_3(t)}_{L^1(\mu_n)}\rightarrow 0$. By the same line of argument, $ \pnorm{I_2(t)}_{L^1(\mu_n)}\conv 0 $ and $ \pnorm{J_2(t)}_{L^1(\mu_n)}\rightarrow 0$.
	
	Write $ \{0,\dots,[nt]-1\}\setminus \mathcal{B}=\bigcup_{i=1}^{k+1} C_i $ where $ C_i$ denotes the intersection of $ \{0,\dots,[nt]-1\} $ with the $ i $th small block for $ 1\le i\le k $. Also, $ C_{k+1} $ denotes the intersection of $ \{0,\dots,[nt]-1\} $ with the last remaining block. Write $ C_i=\{\ell,\ell+1,\dots,u\} $. Then
	\begin{multline*}
		\sum_{0\le r<s<[nt]\colon s\in C_i}v_n\circ \tf_n^r\otimes v_n\circ \tf_n^s=\sum_{r=0}^{\ell-1}\sum_{s=\ell}^{u} v_n\circ \tf_n^r \otimes v_n\circ \tf_n^s\\
		+\sum_{\ell\le r<s\le u}v_n\circ \tf_n^r\otimes v_n\circ \tf_n^s.
	\end{multline*}
	Hence by Lemma~\ref{lemma:moment_bds},
	\begin{flalign}
		\pnorm{\sum_{0\le r<s<[nt]\colon s\in C_i}\!\! v_n\circ \tf_n^r\otimes v_n\circ \tf_n^s}_{L^1(\mu_n)}
		&\le \pnorm{\sum_{r=0}^{\ell-1} v_n\circ \tf_n^r}_{L^2(\mu_n)}\! \pnorm{\sum_{s=\ell}^{u} v_n\circ \tf_n^s}_{L^2(\mu_n)}\nonumber\\
		&\qquad\qquad \quad+\pnorm{\sum_{\ell\le r\le s\le u}v_n\circ \tf_n^r\otimes v_n\circ \tf_n^s}_{L^1(\mu_n)}\nonumber\\
		&\ll \ell^{1/2}\# C_i^{1/2}+\# C_i\ll \# C_i^{1/2}n^{1/2}\label{eq:individual_small_block_bd}.
	\end{flalign}
	Let $ 1\le i\le k $. Then $ \#C_i\le q=[n^b] $. Also, $ k=O(n^{1-a}) $ and $ \# C_{k+1}=O(n^a) $. Thus
	\begin{align*}
		\pnorm{J_3(t)}_{L^1(\mu_n)}&\le \frac{1}{n}\sum_{i=1}^{k+1} \pnorm{\sum_{0\le r<s<[nt]\colon s\in C_i}\!\!v_n\circ \tf_n^r\otimes v_n\circ \tf_n^s}_{L^1(\mu_n)}\\
		&\ll \frac{1}{n}(n^{1-a}n^{\frac{1}{2}(b+1)}+n^{\frac{1}{2}(a+1)})\ll n^{\frac{1}{2}(b+1)-a}+n^{\frac{1}{2}(a-1)}=o(1),
	\end{align*}
	as required.
\end{proof}
For $ 1\le i\le k $ let
\begin{align*}
	X_i&=n^{-1/2}\sum_{0\le r<p}v_n\circ\tf_n^{r+(i-1)(p+q)},\\
	\bbX_i&=n^{-1}\sum_{0\le r<s< p}(v_n\circ \tf_n^r \otimes v_n\circ \tf_n^s) \circ \tf_n^{(i-1)(p+q)}.
\end{align*}
For $ 0\le t\le 1 $ define 
\[ \widetilde W_n(t)=\sum_{1\le i\le[kt]} X_i,\quad \widetilde \bbW_n(t)=\sum_{1\le i<j\le [kt]}X_i\otimes X_j. \]
\begin{prop}\label{prop:W_n_tilde_W_n_approx}
	Suppose that $ a>\frac{b+1}{2} $. Let $ t\in[0,1] $. Then 
	\[ W_n(t)-\widetilde W_n(t)\rightarrow_{\mu_n} 0,\qquad \bbW_n(t)-\widetilde \bbW_n(t)-\sum_{i=1}^{[kt]}\bbX_i\rightarrow_{\mu_n} 0. \]
\end{prop}
\begin{proof}
	Recall the definitions of $ I_1 $ and $ J_1$ from \eqref{eq:iterated_big_block_cont}. Let $ t\in [0,1] $. 
	Since $ [kt](p+q) $ is the first term of the $([kt]+1)$th big block,
	\begin{equation*}
		I_1\left(\frac{[kt](p+q)}{n}\right)=\sum_{1\le i\le[kt]}X_i,\quad J_1\left(\frac{[kt](p+q)}{n}\right)=\sum_{1\le i\le [kt]}\bbX_i+\sum_{1\le i<j\le [kt]}X_i\otimes X_j.
	\end{equation*}
	Hence
	\[ W_n([kt]\tfrac{p+q}{n})=\sum_{1\le i\le[kt]}X_i+I_2([kt]\tfrac{p+q}{n})=\widetilde{W_n}(t)+I_2([kt]\tfrac{p+q}{n}) \]
	and similarly 
	\[  \bbW_n([kt]\tfrac{p+q}{n})=\stilde \bbW_n(t)+\sum_{1\le i\le[kt]}\bbX_i+J_2([kt]\tfrac{p+q}{n})+J_3([kt]\tfrac{p+q}{n}). \]
	By Lemma~\ref{lemma:small_blocks_negligible}, it follows that $ W_n([kt]\frac{p+q}{n})-\stilde W_n(t)\rightarrow_{\mu_n} 0 $ and $ \bbW_n([kt]\tfrac{p+q}{n})-\stilde \bbW_n(t)-\sum_{i=1}^{[kt]}\bbX_i\rightarrow_{\mu_n} 0. $ It remains to show that $ \bbW_n(t)-\bbW_n([kt]\frac{p+q}{n})\rightarrow_{\mu_n} 0$ and $ W_n(t)-W_n([kt]\frac{p+q}{n}) \rightarrow_{\mu_n} 0.$ Let $ 0\le t'\le t $. Let $ C=\{[nt'],\dots,[nt]-1\} $. By~\eqref{eq:individual_small_block_bd},
	\begin{align*}
		\pnorm{\bbW_n(t)-\bbW_n(t')}_{L^1(\mu_n)}&=\pnorm{n^{-1}\sum_{0\le r<s<[nt] \colon s\in C}v_n\circ \tf_n^r \otimes v_n\circ \tf_n^s}_{L^1(\mu_n)}\\
		&\ll n^{-1/2}\# C^{1/2}\ll n^{-1/2}([nt]-[nt'])^{1/2}\\
		&\ll (n^{-1}+t-t')^{1/2}.
	\end{align*}
	Now, $ [kt]\frac{p+q}{n}=\big[\big[\frac{n}{p+q}\big]t\big]\frac{p+q}{n}\rightarrow t $ as $ n\rightarrow \infty $ so $ \pnorm{\bbW_n(t)-\bbW_n([kt]\frac{p+q}{n})}_{L^1(\mu_n)}\rightarrow 0. $ By a similar argument, $ W_n(t)-W_n([kt]\frac{p+q}{n}) \rightarrow_{\mu_n} 0.$
\end{proof}
\begin{lemma}\label{lemma:lln_diagonal}
	Suppose that $ b>\gamma^{-1} $. Let $ t\in [0,1] $. Then $ \sum_{i=1}^{[kt]}\bbX_i\rightarrow_{\mu_n} tE$ as $ n\rightarrow \infty $.
\end{lemma}
\begin{proof}First note that $ [kt]/n\sim t/p $. Hence by Proposition~\ref{prop:uniform_conv_var} and the fact that $ \lim_{n\rightarrow\infty}E_n=E $,
	\begin{align*}
		\lim_{n\rightarrow\infty}\sum_{i=1}^{[kt]}\Ex[\mu_n]{\bbX_i}&=\lim_{n\rightarrow\infty}[kt]\Ex[\mu_n]{\bbX_1}=\lim_{n\rightarrow\infty}\Ex[\mu_n]{\frac{t}{p}\sum_{0\le r<s< p}v_n\circ \tf_n^r \otimes v_n\circ \tf_n^s}\\
		&=\lim_{n\rightarrow \infty}tp^{-1}\Ex[\mu_n]{\bbS_{v_n}(p,n)} = tE.
	\end{align*}
	It remains to show that
	\[ \lim_{n\rightarrow \infty}\Ex[\mu_n]{\Bigg|\sum_{i=1}^{[kt]}(\bbX_i-\Ex[\mu_n]{\bbX_i})\Bigg|}=0. \]
	Write $ \bbX_i(x)=\varPhi_i(\tf_n^{\ell_i}x,\dots,\tf_n^{u_i}x) $ where $ \ell_i=(i-1)(p+q)$, $u_i=\ell_i+p-1$ and
	\[ \varPhi_i(y_0,\dots,y_{u_i-\ell_i})=\frac{1}{n}\sum_{0\le r<s\le u_i-\ell_i}v_n(y_r)\otimes v_n(y_s). \]
	Let $ R=\max_i \norminf{\varPhi_i} $ and define $ F\map{B_{\R^{d\times d}}(0,R)^{[kt]}}{\R} $ by $ F(z_1,\dots,z_{[kt]})=|\sum_{i=1}^{[kt]}(z_i - \Ex[\mu_n]{\bbX_i})| $. 
	
	Let $ (\widehat\bbX_i) $ be independent copies of $ (\bbX_i) $. By Lemma~\ref{lemma:fcb_weak_dep_multidim},
	\[ \Ex[\mu_n]{\bigg|\sum_{i=1}^{[kt]}(\bbX_i-\Ex[\mu_n]{\bbX_i})\bigg|}\le A+\Ex{\bigg|\sum_{i=1}^{[kt]}(\widehat\bbX_i-\Ex[\mu_n]{\bbX_i})\bigg|} \]
	where 
	\[ |A|\le C\sum_{r=1}^{[kt]-1}(\ell_{r+1}-u_r)^{-\gamma}\biggl(\norminf{F}+\Lip(F)\sum_{i=1}^{[kt]}\sum_{j=0}^{u_i-\ell_i}\sdsemi{\varPhi_i}{j}\biggr). \]
	Note that $ \norminf{\varPhi_i}\le \frac{p^2}{n}\norminf{v_n}^2 $. By a similar calculation to the bound on $ \sdsemi{\Phi_i}{j} $ in the proof of \cite[Lemma 5.5]{fleming2022}, $ \sdsemi{\varPhi_i}{j}\le \frac{1}{n}(u_i-\ell_i)\dhnorm{v_n}^2=\frac{p-1}{n}\dhnorm{v_n}^2$. 
	
	Let $ z=(z_1,\dots,z_{[kt]}), z'=(z'_1,\dots,z'_{[kt]})\in (\R^{d\times d})^{[kt]} $. Then
	\[ |F(z)-F(z')|=\bigg|\sum_{i=1}^{[kt]}z_i-z'_i \bigg|\le \sum_{i=1}^{[kt]}|z_i-z'_i|=|z-z'| \]
	so $ \Lip(F)\le 1 $. Moreover, 
	\[ \norminf{F}\le \sum_{i=1}^{[kt]}(R+|\Ex[\mu_n]{\bbX_i}|)\le \sum_{i=1}^{[kt]}(R+\norminf{\varPhi_i})\le \frac{2kp^2}{n}\norminf{v_n}^2 . \]
	Now $ pk\le n $ and $ \ell_{r+1}-u_r =  q+1\ge n^{b} $ so
	\begin{align*}
		|A|&\le Ckq^{-\gamma}\left(\frac{2kp^2}{n}\norminf{v_n}^2+\frac{kp^2}{n}\dhnorm{v_n}^2\right)\ll \frac{k^2 p^2}{n} q^{-\gamma}\le  nq^{-\gamma}\\
		&\ll n^{1-b\gamma}=o(1).
	\end{align*}
	It remains to prove that $ \Ex{|\sum_{i=1}^{[kt]}(\widehat\bbX_i-\Ex[\mu_n]{\bbX_i})|}\rightarrow 0 $. Without loss of generality take $ \gamma\le 2. $ By von Bahr-Esseen's inequality,
	\[ \Ex{\bigg|\sum_{i=1}^{[kt]}(\widehat\bbX_i-\Ex[\mu_n]{\bbX_i})\bigg|}\le \pnorm{\sum_{i=1}^{[kt]}(\widehat{\bbX}_i-\Ex[\mu_n]{\bbX_i})}_\gamma
	\ll\left(\sum_{i=1}^{[kt]}\pnorm{\widehat{\bbX}_i-\Ex{\bbX_i}}_\gamma^\gamma\right)^{1/\gamma}. \]
	Now by Lemma~\ref{lemma:moment_bds}, 
	\begin{align*}
		\pnorm{\widehat{\bbX}_i-\Ex[\mu_n]{\bbX_i}}_\gamma&=\pnorm{\bbX_i-\Ex[\mu_n]{\bbX_i}}_\gamma\le 2\pnorm{\bbX_i}_\gamma\\
		&=\frac{2}{n}\pnorm{\sum_{0\le r<s< p}v_n\circ \tf_n^r \otimes v_n\circ \tf_n^s}_\gamma=O(p/n)
	\end{align*}
	so
	\[ \Ex{\bigg|\sum_{i=1}^{[kt]}(\widehat\bbX_i-\Ex[\mu_n]{\bbX_i})\bigg|}\ll ([kt](p/n)^{\gamma})^{1/\gamma}\ll k^{(1-\gamma)/\gamma}=o(1), \]
	as required.
\end{proof}
\begin{lemma}\label{lemma:iterated_clt_non_diagonal}
	Suppose that $ a+\gamma b>2$. Then
	\[ (\widetilde W_n,\widetilde \bbW_n)\rightarrow_{\mu_n} \left(W,\int W\otimes dW\right) \text { in }D([0,1],\R^d\times \R^{d\times d}). \]
\end{lemma}
\begin{proof}
	By the portmanteau theorem \cite[p.53]{bogachev_weak_conv}, it suffices to show that 
	\[ \lim_{n\rightarrow \infty}\Ex[\mu_n]{G(\widetilde W_n,\widetilde \bbW_n)}=\Ex{G\left(W,\int W\otimes dW\right)} \]
	for all Lipschitz, bounded $ G\map{D([0,1],\R^d\times \R^{d\times d})}{\R} $. 
	
	Note that
	\[ X_i(x)=\varPhi_i(\tf_n^{\ell_i}x,\dots,\tf_n^{u_i}x) \]
	where $ \ell_i=(i-1)(p+q)$, $u_i=\ell_i+p-1 $ and \[ \varPhi_i(y_0,\dots,y_{u_i-\ell_i})=n^{-1/2}\sum_{r=0}^{u_i-\ell_i}v_n(y_r). \]
	Let $ 0\le r\le u_i-\ell_i $. Then $ \sdsemi{\varPhi_i}{r}\le n^{-1/2}\dsemi{v_n}.$ Let $ R=\max_i \norminf{\varPhi_i}\le pn^{-1/2}\norminf{v_n}.$ Define $ \pi_k\map{B(0,R)^k}{D([0,1],\R^d\times \R^{d\times d})} $ by
	\[ \pi_k(x_1,\dots,x_k)(t)=\left(\sum_{i=1}^{[kt]}x_i,\sum_{1\le i<j\le {[kt]}}x_i\otimes x_j\right) \]
	for $ t\in [0,1] $. Then $ (\stilde W_n,\stilde \bbW_n)=\pi_k(X_1,\dots,X_k) $. 
	
	Now for all $ (x_1,\dots,x_k),(x'_1,\dots,x'_k)\in B(0,R)^k $,
	\[ \sup_{t\in[0,1]}\Bigg|\sum_{i=1}^{[kt]}x_i-\sum_{i=1}^{[kt]}x'_i\Bigg|\le \sum_{i=1}^k|x_i-x_i'| \]
	and 
	\begin{align*}
		\sup_{t\in[0,1]}\Bigg|\sum_{1\le i<j\le {[kt]}}(x_i\otimes x_j-x'_i\otimes x'_j)\Bigg|
		&\le\sum_{1\le i<j\le k}|x_i \otimes x_j-x'_i\otimes x'_j|\\
		&\le \sum_{1\le i<j\le k}|x_i\otimes (x_j-x'_j)|+|(x_i-x'_i)\otimes x'_j|\\
		&\le 2kR\sum_{j=1}^{k}|x_j-x'_j|.
	\end{align*}
	Thus $ \Lip(\pi_k)\le 1+2kR. $
	
	Let $ (\shat X_i) $ be independent copies of $ (X_i) $ and define 
	\[  (\shat W_n,\shat \bbW_n)(t)=\pi_k(\shat X_1,\dots,\shat X_k)(t)=\left(\sum_{1\le i\le[kt]}\shat X_i,\sum_{1\le i<j\le {[kt]}}\shat X_i\otimes \shat X_j\right).  \]
	By Lemma~\ref{lemma:fcb_weak_dep_multidim},
	$|\Ex[\mu_n]{G(\stilde W_n,\stilde \bbW_n)}-\Ex{G(\shat W_n,\shat \bbW_n)}|\le A,$
	where
	\begin{align*}
		A &= C\sum_{r=1}^{k-1}(\ell_{r+1}-u_r)^{-\gamma}\biggl(\norminf{G\circ \pi_k}+\Lip(G\circ \pi_k)\sum_{i=1}^{k}\sum_{j=0}^{u_i-\ell_i}\sdsemi{\varPhi_i}{j}\biggr)\\
		&\le C\sum_{r=1}^{k-1}(\ell_{r+1}-u_r)^{-\gamma}\biggl(\norminf{G}+(1+2kR)\Lip(G)\sum_{i=1}^{k}\sum_{j=0}^{u_i-\ell_i}\sdsemi{\varPhi_i}{j}\biggr)\\
		&\le Ck(q+1)^{-\gamma}(\norminf{G}+(1+2kpn^{-1/2}\norminf{v_n})\Lip(G) kp n^{-1/2}\dsemi{v_n})\\
		&\ll n^{1-a}n^{-b\gamma}n=n^{2-a-b\gamma}=o(1).
	\end{align*}
	
	It remains to show that 
	\[(\shat W_n,\shat\bbW_n)\rightarrow_{\mu_n} \left(W,\int W\otimes dW\right) \text { in }D([0,1],\R^d\times \R^{d\times d}).  \]
	Indeed, once we have proved this it follows that
	\begin{align*}
		\lim_{n\rightarrow \infty}\Ex[\mu_n]{G(\stilde W_n,\stilde\bbW_n)}=\lim_{n\rightarrow \infty}\Ex{G(\shat W_n,\shat\bbW_n)}=\Ex{G\left(W,\int W\otimes dW\right)},
	\end{align*}
	completing the proof of this lemma.
	
	Note that $ n/k\sim p $. Thus by Proposition~\ref{prop:uniform_conv_var} and the fact that $ \lim_{n\rightarrow\infty}\Sigma_n=\Sigma $, 
	\begin{align*}
		\lim_{n\rightarrow \infty}\sum_{i=1}^{[kt]}\Ex{\shat X_i \otimes \shat X_i}&=\lim_{n\rightarrow \infty}[kt]\Ex{\shat X_1\otimes \shat X_1}\\
		&=\lim_{n\rightarrow \infty}\frac{[kt]}{n}\Ex[\mu_n]{\sum_{r=0}^{p-1}v_n\circ \tf_n^r\otimes \sum_{r=0}^{p-1}v_n\circ \tf_n^r}\\
		&=\lim_{n\rightarrow \infty}\frac{t}{p}\Ex[\mu_n]{S_{v_n}(p,n)\otimes S_{v_n}(p,n)}=\Sigma t.
	\end{align*}
	Hence hypothesis (i) of Lemma~\ref{lemma:lindeberg_wip} is satisfied with $ \chi_{n,i}=\shat X_i $. Now by Lemma~\ref{lemma:moment_bds},
	\begin{align*}
		\sum_{i=1}^{k}\Ex{|\shat X_i|^{2\gamma}}&=kn^{-\gamma}\Ex[\mu_n]{\bigg|\sum_{r=0}^{p-1}v_n\circ\tf_n^r\bigg|^{2\gamma}}\ll kn^{-\gamma}p^\gamma \\
		&\le k(kp)^{-\gamma}p^\gamma=k^{1-\gamma}=o(1)
	\end{align*}
	so hypothesis~(ii) of Lemma~\ref{lemma:lindeberg_wip} is satisfied with $ p=\gamma $. This completes the proof.
\end{proof}
We are now ready to prove the iterated WIP (Theorem~\ref{thm:iterated_wip}).
\begin{proof}[Proof of Theorem~\ref{thm:iterated_wip}]
	Since $ \gamma>1 $, we can choose $ 0<b<a<1 $ such that $ b>\gamma^{-1} $, $ a>\frac{b+1}{2} $ and $ a+\gamma b>2 $. Then the conditions of Proposition~\ref{prop:W_n_tilde_W_n_approx} and Lemmas~\ref{lemma:lln_diagonal} and~\ref{lemma:iterated_clt_non_diagonal} are satisfied. Let $ 0\le t_1, t_2, \dots , t_\ell\le 1$, $ \ell\ge 1 $. Write
	\[ \big((W_n,\bbW_n)(t_1),(W_n,\bbW_n)(t_2),\dots,(W_n,\bbW_n)(t_\ell)\big)=K_1+K_2+K_3, \]
	where
	\begin{align*}
		K_1&=(A(t_1),A(t_2),\dots,A(t_\ell)),\\
		K_2&=\Biggl(\biggl(0,\sum_{i=1}^{[kt_1]}\bbX_i\biggr),\biggl(0,\sum_{i=1}^{[kt_2]}\bbX_i\biggr),\dots,\biggl(0,\sum_{i=1}^{[kt_{\ell}]}\bbX_i\biggr)\Biggr),\\
		K_3&=\big((\stilde W_n,\stilde\bbW_n)(t_1),(\stilde W_n,\stilde \bbW_n)(t_2),\dots,(\stilde W_n,\stilde \bbW_n)(t_\ell)\big).
	\end{align*}
	Here $ A(t)=(W_n(t)-\stilde W_n(t),\bbW_n(t)-\stilde\bbW_n(t)-\sum_{i=1}^{[kt]}\bbX_i). $
	
	By Proposition~\ref{prop:W_n_tilde_W_n_approx}, $ K_1\rightarrow_{\mu_n} 0 $ and by Lemma~\ref{lemma:lln_diagonal}, 
	\[ K_2\rightarrow_{\mu_n} ((0,t_1 E),(0,t_2 E),\dots,(0,t_\ell E)). \]
	Moreover, by Lemma~\ref{lemma:iterated_clt_non_diagonal},
	\[ K_3 \rightarrow_{\mu_n}\! \Big(\Big(W(t_1),\int_0^{t_1}W\otimes dW\Big),\Big(W(t_2),\int_0^{t_2}W\otimes dW\Big),\ldots,\Big(W(t_\ell),\int_0^{t_\ell}W\otimes dW\Big)\Big). \]
	Hence by Slutsky's theorem,
	\[ K_1+K_2+K_3 \rightarrow_{\mu_n}\big((W,\bbW)(t_1),(W,\bbW)(t_2),\dots,(W,\bbW)(t_\ell)\big), \]
	as required.
\end{proof}
\subsection{Proof of Theorem~\ref{thm:homog_from_fcb_family}}\label{subsection:pf_main_result}
We now have all the ingredients needed to prove Theorem~\ref{thm:homog_from_fcb_family}. 
\begin{proof}[Proof of Theorem~\ref{thm:homog_from_fcb_family}]
We proceed by applying~\cite[Theorem~2.17]{chevyrev2022deterministic}, so we need to check Assumptions 2.11 and 2.12 from~\cite{chevyrev2022deterministic}. 

By the arguments in the proof of~\cite[Proposition 3.9]{korepanov2022deterministic}, Assumptions 2.11 and 2.12(ii)(a) follow from~\eqref{eq:conv_correlation_fns} and Lemma~\ref{lemma:moment_bds}. Since $ \mu_n $ is $ T_n $-invariant for all $ n $, Assumption 2.12(i) also follows from Lemma~\ref{lemma:moment_bds} (cf.\ \cite[Remark 2.13]{chevyrev2022deterministic}). 

It remains to verify Assumption 2.12(ii)(b). Let $ v_n\in \holsp(M,\R^d) $, $ n\in \N\cup\{\infty\} $, with $ \Ex[\mu_n]{v_n}=0 $ and $\sup_{n\ge1}\dhnorm{v_n}<\infty$. We assume that $ \lim_{n\to \infty}\norminf{v_n-v_\infty}=0$. Define $ \Sigma_n $ and $ E_n $ as in~\eqref{eq:covar_drift_formulas}. It suffices to prove that $ \lim_{n\to\infty}\Sigma_n=\Sigma_\infty $ and $ \lim_{n\to\infty}E_n=E_\infty $. Assumption 2.12(ii)(b) then follows from Theorem~\ref{thm:iterated_wip}, with
\[ \mathfrak{B}_1(v,w)=\Ex[\mu_\infty]{vw},\quad \mathfrak{B}_2(v,w)=\sum_{\ell\ge 1}^\infty \Ex[\mu_\infty]{vw\circ T_\infty^\ell}.  \]
We show that $ \lim_{n\to\infty}E_n=E_\infty $; the proof that $ \lim_{n\to\infty}\Sigma_n = \Sigma_\infty $ is similar. By Proposition~\ref{prop:uniform_conv_var}, the series $ E_n = \sum_{\ell\ge 1}\Ex[\mu_n]{v_n\otimes v_n\circ \tf_n^\ell} $ is convergent uniformly in $ n $. Hence it suffices to show that $ \lim_{n\rightarrow \infty}\Ex[\mu_n]{v_n\otimes v_n\circ \tf_n^\ell}=\Ex[\mu_\infty]{v_\infty\otimes v_\infty\circ \tf_\infty^\ell} $ for each fixed $ \ell\ge 1 $. Write $ \Ex[\mu_n]{v_n\otimes v_n\circ \tf_n^\ell}-\Ex[\mu_\infty]{v_\infty\otimes v_\infty\circ \tf_\infty^\ell}=A(n)+B(n) $, where
\begin{align*}
	A(n)&=\Ex[\mu_n]{v_n\otimes v_n\circ \tf_n^\ell-v_\infty\otimes v_\infty\circ \tf_n^\ell},\\
	 \quad B(n)&=\Ex[\mu_n]{v_\infty\otimes v_\infty\circ \tf_n^\ell}-\Ex[\mu_\infty]{v_\infty\otimes v_\infty\circ \tf_\infty^\ell}.
\end{align*}
Now, 
\[ \norminf{v_n\otimes v_n\circ \tf_n^\ell-v_\infty\otimes v_\infty\circ \tf_n^\ell}\le \norminf{v_n-v_\infty}\norminf{v_n}+\norminf{v_\infty}\norminf{v_n-v_\infty} \]
so $ A(n)\to 0 $ as $ n\to\infty $. Finally, by~\eqref{eq:conv_correlation_fns}, we have $ B(n)\to 0 $.
\end{proof}
%
%By (A1), $ \mu_n\wconv \mu_\infty $ and by~\cite[Corollary 2.2.10]{bogachev_weak_conv}, it follows that $ \int f\, d\mu_n\to \int f\, d\mu_\infty $ whenever $ f $ is bounded, measurable and continuous $ \mu_\infty $-a.e. Fix $ j\in \N $ and let $ v,w:M\to \R $ be continuous and bounded. By (A1), $ T_\infty $ is continuous $ \mu_\infty $-a.e. Since $ \mu_\infty $ is $ T_\infty $-invariant, it follows that $ vw\circ T_\infty^j $ is continuous $ \mu_\infty $-a.e. Hence $ \int vw\circ T_\infty^j d\mu_n\to \int vw\circ T_\infty^j d\mu_\infty $. By combining this with (A2), it follows that condition (3.2) in~\cite{korepanov2022deterministic} is satisfied.
%
%As shown in~\cite[Sect.~3.3]{korepanov2022deterministic}, Assumptions 2.11 and 2.12(ii)(a) follow from~\cite[(3.2)]{korepanov2019explicit}. 
\section{Families of nonuniformly hyperbolic maps}\label{section:families_nuh_maps}
In this section, we consider uniform families of nonuniformly hyperbolic maps. Our main result, Theorem~\ref{thm:uniform_fcb}, is that such families satisfy the \fcb{} with uniform rate. In Subsection~\ref{subsection:nuh_map_def}, we recall the definition of a nonuniformly hyperbolic map. In Subsection~\ref{subsection:explicit_fcb}, we prove some estimates for a fixed nonuniformly hyperbolic map. In Subsection~\ref{subsection:uniform_families}, we describe our uniformity criteria for families of nonuniformly hyperbolic maps and prove Theorem~\ref{thm:uniform_fcb}.
\subsection{Nonuniformly hyperbolic maps}\label{subsection:nuh_map_def}
In this subsection, we recall the notion of a nonuniformly hyperbolic map in the sense of Young \cite{young1998statistical,young1999recurrence}. The definition we use is based on \cite{korepanov2019explicit} and \cite{bmtmaps}. In particular, we do not assume uniform contraction along stable manifolds.

\quad\textbf{Gibbs-Markov maps}\quad Let \((\bar Y,\bar\mu_Y)\) be a probability space and let $\bar F\map{\bar Y}{\bar Y}$ be ergodic and measure-preserving. Let $\alpha$ be an at most countable, measurable partition of $\bar Y$. We assume that there exist constants $D_0>0,\ \theta\in (0,1)$ such that for all elements $a\in \alpha$:
\begin{itemize}
	\item (Full-branch condition) The map $\bar F|_{a}\map{a}{\bar Y}$ is a measurable bijection.
	\item For all distinct $y,y'\in \bar Y$ the separation time 
	\begin{equation*}
			s(y,y')=\inf\{n\ge 0: \bar F^n y,\, \bar F^n y' \text{ lie in distinct elements of }\alpha\}<\infty.
		\end{equation*}
	\item \label{item:GM_bdd_distortion} Define $\zeta\map{a}{\R^+}$ by $\zeta=d\bar\mu_Y/(d\, (F|_a^{-1})_* \bar\mu_Y)$. We have $ |\log \zeta(y)-\log \zeta(y')|\le D_0\theta^{s(y,y')} $ for all $ y,y'\in a $.
\end{itemize}
Then we call $\bar{F}\map{\bar Y}{\bar Y}$ a full-branch Gibbs-Markov map.

\quad\textbf{Two-sided Gibbs-Markov maps}\quad Let $(Y,d)$ be a bounded metric space with Borel probability measure $\mu_Y$ and let $F\map{Y}{Y}$ be ergodic and measure-preserving. Let $\bar{F}\map{\bar{Y}}{\bar{Y}}$ be a full-branch Gibbs-Markov map with associated measure $\bar\mu_Y$.

We suppose that there exists a measure-preserving semi-conjugacy $\bar{\pi}\map{Y}{\bar{Y}}$, so $\bar{\pi}\circ F=\bar{F}\circ \bar{\pi}$ and $\bar{\pi}_{*}\mu_Y=\bar{\mu}_Y.$ The separation time $s(\cdot,\cdot)$ on $\bar Y$ lifts to a separation time on $Y$ given by $s(y,y')=s(\bar\pi y,\bar\pi y')$. Suppose that there exist constants $D_0>0$, $\theta\in(0,1)$ such that
\begin{equation}\label{eq:two_sided_GM_contraction}
	d(F^k y,F^k y')\le D_0(\theta^k+\theta^{s(y,y')-k}) \text{ for all }y,y'\in Y,k\ge 0.
\end{equation}
Then we call $F\map{Y}{Y}$ a \textit{two-sided Gibbs-Markov map.}

\quad\textbf{One-sided Young towers}\quad Let $\bar\phi\map{\bar Y}{\Z^+}$ be integrable and constant on partition elements of $\alpha$. We define the one-sided Young tower $\bar\Delta=\bar Y^{\bar\phi}$ and tower map $\bar f\map{\bar \Delta}{\bar \Delta}$ by  
\begin{equation}\label{eq:one-sided_yt_def}
	\bar\Delta=\{(\bar y,\ell)\in \bar Y\times \Z: 0\le \ell<\bar \phi(y)\},\ \bar{f}(\bar y,\ell)=\begin{cases}
			(\bar y,\ell+1),& \ell<\bar\phi(y)-1,\\
			(\bar{F}\bar{y},0),& \ell=\bar\phi(y)-1.
		\end{cases}
\end{equation}
We extend the separation time $s(\cdot,\cdot)$ to $\bar\Delta$ by defining
\[s((\bar y,\ell),(\bar y',\ell'))=\begin{cases}
		s(\bar{y},\bar{y}'),& \ell=\ell',\\
		0,& \ell\ne \ell'.
	\end{cases}\]
	Note that for $\theta\in (0,1)$ we can define a metric by $d_\theta(\bar{p},\bar{q})=\theta^{s(\bar{p},\bar{q})}$.
	
	Now, $\bar{\mu}_\Delta=(\bar{\mu}_Y \times \text{counting})/\int_{\bar{Y}}\bar{\phi} d\bar{\mu}_Y$ is an ergodic $\bar{f}$-invariant probability measure on $\bar{\Delta}$.
	\vspace{1em}
	
	\textbf{Two-sided Young towers}\quad Let $F\map{Y}{Y}$ be a two-sided Gibbs-Markov map and let $\phi\map{Y}{\Z^+}$ be an integrable function that is constant on $\bar{\pi}^{-1}a$ for each $a\in \alpha$. In particular, $\phi$ projects to a function $\bar{\phi}\map{\bar{Y}}{M}$ that is constant on partition elements of $\alpha$.
	
	Define the one-sided Young tower $\bar{\Delta}=\bar{Y}^{\bar{\phi}}$ as in \eqref{eq:one-sided_yt_def}. Using $\phi$ in place of $\bar{\phi}$ and $F\map{Y}{Y}$ in place of $\bar{F}\map{\bar{Y}}{\bar{Y}}$, we define the \textit{two-sided Young tower} $\Delta=Y^{\phi}$ and tower map $f\map{\Delta}{\Delta}$ in the same way. Likewise, we define an ergodic $ f $-invariant probability measure on $\Delta$ by $\mu_\Delta=(\mu_Y \times \text{counting})/\int_{Y}\phi\, d\mu_Y$. 
	
	We extend $\bar{\pi}\map{Y}{\bar{Y}}$ to a map $\bar{\pi}\map{\Delta}{\bar{\Delta}}$ by setting $\bar{\pi}(y,\ell)=(\bar{\pi}y,\ell)$ for all $(y,\ell)\in \Delta$. Note that $\bar{\pi}$ is a measure-preserving semi-conjugacy; $\bar{\pi}\circ f=\bar{f}\circ \bar{\pi}$ and $\bar{\pi}_{*}\mu_\Delta=\bar{\mu}_\Delta$. The separation time $s$ on $\bar{\Delta}$ lifts to $\Delta$ by defining $s(y,y)=s(\bar{\pi}y,\bar{\pi}y').$
	
	Let $ \tf\map{\ms}{\ms} $ be a measure-preserving transformation on a probability space $ (\ms,\mu) $. Suppose that there exists $Y\subset M$ measurable with $\mu(Y)>0$ such that:
	\begin{itemize}
		\item $F=T^{\phi}\map{Y}{Y}$ is a two-sided Gibbs-Markov map with respect to some probability measure $\mu_Y$.
		\item $\phi$ is constant on partition elements of $\bar{\pi}^{-1}\alpha$, so we can define Young towers $\Delta=Y^\phi$ and $\bar\Delta=\bar{Y}^{\bar{\phi}}$.
		\item There exist constants $ D_0>0 $ and $ \theta\in(0,1) $ such that for all $ y,y'\in Y $, $ 0\le \ell <\phi(y)$ we have
		\begin{equation}\label{eq:yt_intermediate_iterate_bd}
				d(\tf^\ell y,\tf^\ell y')\le D_0(d(y,y')+\theta^{s(y,y')})
			\end{equation}
		\item The map $\pi_M\map{\Delta}{M}$, $\pi_M(y,\ell)=T^\ell y$ is a measure-preserving semiconjugacy.
	\end{itemize}
	Then we call $\tf\map{\ms}{\ms}$ a \textit{nonuniformly hyperbolic map.}
	\begin{remark}
		Note that we have not assumed that $ \gcd\{\phi(y)\colon y\in Y\}=1 $. By assuming that $ \mu $ is mixing we are able to reduce to this case, as we briefly explain in the next subsection.
	\end{remark}
	\subsection{Explicit estimates for nonuniformly hyperbolic maps}\label{subsection:explicit_fcb}
	In this subsection, we consider a fixed nonuniformly hyperbolic map $ \tf\map{\ms}{\ms} $. We assume that $ \tf $ is mixing and that there exist constants $ \beta>1 $ and $ C_\phi>0 $ such that $ \mu_Y(\phi\ge k)\le C_\phi k^{-\beta} $ for all $ k\ge 1 $.
	\begin{defn}
		For $ v\map{\ms}{\R} $ define
		\[ [v]_{\mathcal{H}}=\sup_{y\ne y'}\sup_{0\le \ell<\phi(y)}\frac{|v(\tf^\ell y)-v(\tf^\ell y')|}{d(y,y')+\theta^{s(y,y')}}. \]
	\end{defn}
	\begin{prop}\label{prop:holder_implies_dyn_holder}
		Let $ v\map{\ms}{\R} $ be Lipschitz. Then $ [v]_{\mathcal{H}}\le D_0 \Lip(v) $.
	\end{prop}
	\begin{proof}
		Let $ y,y'\in Y $, $ 0\le \ell<\phi(y) $. By \eqref{eq:yt_intermediate_iterate_bd},
		\begin{align*}
				|v(\tf^\ell y)-v(\tf^\ell y')|&\le \Lip(v) d(\tf^\ell y,\tf^\ell y')\\
				&\le D_0 \Lip(v)(d(y,y')+\theta^{s(y,y')}).\qedhere
			\end{align*}
	\end{proof}
	Let $ q\ge 1 $. Given a function $ G\map{\ms^q}{\R} $ and $ 0\le i<q $ we denote
	\[ [G]_{\mathcal{H},i}=\sup_{x_0,\dots,x_{q-1}\in \ms}[G(x_0,\dots,x_{i-1},\cdot,x_{i+1},\dots,x_{q-1})] _{\mathcal{H}}. \]
	We call $ G $ separately dynamically \holder{} if $ \norminf{G}+\sum_{i=0}^{q-1}[G]_{\mathcal{H},i}<\infty $.
	
	We are now ready to state the main result of this subsection. Note that there exist $ \delta>0 $ and a finite set $ I\subset \mathbb{N} $ with $ \gcd \{I\}=\gcd\{\phi(y):y\in Y\} $ and $ {\mu_Y(\phi=k)}\ge \delta $ for all $ k\in I $.
	\begin{lemma}\label{lemma:fcb_explicit}There exists a constant $ C>0 $ depending continuously on $ D_0,\theta, \delta$, $\max\{I\}, \beta $ and $ C_\phi$ such that 
		for all $ 0\le p<q $, $ 0\le k_0\le \cdots\le k_{q-1} $,
		\begin{align}\label{eq:dyn_holder_fcb}
			&\bigg|\int_\ms  G(\tf^{k_0}x,\dots,\tf^{k_{q-1}}x)d\mu(x)\nonumber\\
			&\qquad\qquad\qquad -\int_{\ms^2} G(\tf^{k_0}x_0,\dots,\tf^{k_{p-1}}x_0,\tf^{k_p}x_1,\dots,\tf^{k_{q-1}} x_1)d\mu(x_0)d\mu(x_1)\bigg|	\nonumber\\
			&\qquad\le C(k_p-k_{p-1})^{-(\beta-1)}\biggl(\norminf{G}+\sum_{i=0}^{q-1} [G]_{\mathcal{H},i}\biggr)
			\end{align}
		for any separately dynamically \holder{} $ G\map{\ms^q}{\R} $.
	\end{lemma}
In \cite[Theorem~2.3]{fleming2022} we proved the estimate \eqref{eq:dyn_holder_fcb} without showing that $ C $ depends continuously on the constants mentioned above. We now briefly outline how the arguments in~\cite{fleming2022} can be modified to prove Lemma~\ref{lemma:fcb_explicit}.

In \cite[Sect.~3.2]{fleming2022} the proof of~\eqref{eq:dyn_holder_fcb} is reduced to the case where $ \gcd\{\phi(y):y\in Y\}=1 $. It is easy to see that the implicit constant in \cite[Sect.~3.2]{fleming2022} depends continuously on $ \beta $ and $ \gcd\{\phi(y):y\in Y\}\le \max\{I\} $. Similarly, it is straightforward to check that most of the constants that appear in~\cite[Sects.~3.3 and 3.4]{fleming2022} can be written explicitly and vary continuously with $ D_0,\theta, \beta $ and $ C_\phi $. The sole exception to this is \cite[Lemma~3.4]{fleming2022}\footnote{There is a typo in \cite[Lemma~3.4]{fleming2022} where $ \norminf{\indicd{\bar\Delta_0}L^n v-\int_{\bar\Delta} v\, d\bar\mu_\Delta} $ should be replaced by $ \norminf{\indicd{\bar\Delta_0}\big(L^n v-\int_{\bar\Delta} v\, d\bar\mu_\Delta\big)} $.}. (Indeed, the proof of that lemma uses operator renewal theory, so it is unclear how the bound obtained depends on the data associated with $ \tf $.)

	Let $ L $ denote the transfer operator corresponding to $ f:\bar\Delta\to \bar\Delta $. Let $ \bar{\Delta}_0=\{(y,\ell)\in\bar\Delta\colon\ell=0\} $ denote the base of $ \bar\Delta $. For any $ d_\theta $-Lipschitz $ v\map{\bar\Delta}{\R} $ write $ \norm{v}_\theta=\norminf{v}+\Lip(v) $. We now complete the proof of Lemma~\ref{lemma:fcb_explicit} by providing the following uniform version of \cite[Lemma~3.4]{fleming2022}:
	\begin{lemma}Suppose that $ \gcd\{\phi(y)\colon y\in Y\}=1 $. 
	Then there exists a constant $ C>0 $ depending continuously on $ D_0,\theta, \delta, \max\{I\}, \beta $ and $ C_\phi$ such that for all $ d_\theta $-Lipschitz $ v\map{\bar\Delta}{\R} $ and for any $ k\ge 1 $,
	\[ \norminf{\indicd{\bar\Delta_0}\biggl(L^k v-\int v\, d\bar\mu_\Delta\biggr)}\le C \norm{v}_\theta k^{-(\beta-1)}. \]
\end{lemma}
	\begin{proof}
	Recall that $ L $ is given pointwise by 
	\[(Lv)(x)=\sum_{\bar fz=x}g(z)v(z), \text{ where } g(y,\ell)=\begin{cases}
		\zeta(y),& \ell=\phi(y)-1,\\
		1,& \ell<\phi(y)-1
	\end{cases}.\]
	Hence $ (L^k v)(x)=\sum_{\bar f^k z=x}g_k(z)v(z) $, where $ g_k = \prod_{i=0}^{k-1}g \circ \bar f^i $. 
	
	Throughout this proof, $ C>0 $ denotes a constant depending continuously on $ D_0,\,\theta,\, \delta,\, \max\{I\},\, \beta $ and $ C_\phi$. We first show that for all $ d_\theta $-Lipschitz $ w\map{\bar\Delta}{\R} $ and all $ n\ge 0 $,
	\begin{equation}\label{eq:sup_transfer_op_base_bd}
		\sup_{\bar\Delta_0}|L^n w|\le C(\Lip(w)|\theta^{\psi_n}|_1+|w|_1),
	\end{equation}
	where $ \psi_n(x)= \#\{j=1,\dots,n\colon \bar f^j x\in \bar\Delta_0 \}$ denotes the number of returns to $ \bar\Delta_0 $ by time $ n $.
	
	Let $ x,x'\in \bar\Delta_0 $. Then we can pair preimages $ z,z' $ of $ x,x' $ so that $ \bar f^j z, \bar f^j z' $ lie in the same partition element of $ \bar\Delta $ for $ 0\le j<n $, so $ s(z,z')=\psi_n(z')+s(x,x') $. Write $ (L^n w)(x)-(L^n w)(x')=I_1+I_2 $, where
	\begin{align*}
		I_1=\sum_{\bar f^n z=x}g_n(z)(w(z)-w(z')), \quad I_2=\sum_{\bar f^n z'=x'}w(z')(g_n(z)-g_n(z')).
	\end{align*}
	Note that
	\begin{align*}%\label{eq:transfer_op_lipschitz_part}
		|I_1|\le \Lip(w)\sum_{\bar f^n z=x}g_n(z)\theta^{\psi_n(z')}d_\theta(x,x')=\Lip(w)(L^n \theta^{\psi_n})(x')d_\theta(x,x').
	\end{align*}
	By bounded distortion (see for example~\cite[Proposition~5.2]{korepanov2019explicit}),
	\begin{align*}%\label{eq:transfer_op_bdd_distortion_part}
		|I_2|\le C\sum_{\bar f^n z'=x'}|w(z')|g_n(z')d_\theta(x,x')=C(L^n |w|)(x')d_\theta(x,x').
	\end{align*}
	It follows that
	$ |(L^n w)(x)|\le |(L^n w)(x')|+ \Lip(w)(L^n \theta^{\psi_n})(x') +C(L^n |w|)(x'). $
	Hence integrating over $ x'\in \bar\Delta_0 $ gives
	\begin{align*}
		|(L^n w)(x)|&\le \bar\mu_\Delta(\bar\Delta_0)^{-1}\int_{\bar\Delta_0}(|L^n w|+\Lip(w)|L^n\theta^{\psi_n}|+CL^n|w|)d\bar\mu_\Delta\\
		&\le \bar\mu_\Delta(\bar\Delta_0)^{-1}\bigl((1+C)\pnorm{w}_1+\Lip(w)\pnorm{\theta^{\psi_n}}_1\bigr).
	\end{align*}
	The proof of \eqref{eq:sup_transfer_op_base_bd} follows by noting that 
	$ \bar\mu_\Delta(\bar\Delta_0)^{-1}=\int_{\bar Y} \phi\, d\bar\mu_Y\le C_\phi \sum_{k\ge 1}k^{-\beta} $.
%	
%	Next we show that
%	\begin{equation}\label{eq:transfer_op_lipschitz_bd}
%		\Lip(L^n w)\le C\norm{w}_\theta \text{ for all }n\ge 0.
%	\end{equation}
%	Let $ x,x'\in \bar\Delta $. If $ d_\theta(x,x')=1 $, then
%	\begin{align}\label{eq:transfer_op_trivial_bd}
%		|(L^n w)(x)-(L^n w)(x')|\le 2\norminf{w}=2\norminf{w}d_\theta(x,x').
%	\end{align}
%	Otherwise, we can pair preimages $ z,z' $ of $ x,x'$. Then by \eqref{eq:transfer_op_lipschitz_part} and \eqref{eq:transfer_op_bdd_distortion_part},
%	\begin{align}
%		|(L^n w)(x)-(L^n w)(x')|&\le|I_1|+|I_2|\le d_\theta(x,x')(\Lip(w)+C\norminf{w} )\label{eq:transfer_op_paired_bd}.
%	\end{align}
%	The bound \eqref{eq:transfer_op_lipschitz_bd} follows by combining \eqref{eq:transfer_op_trivial_bd} and \eqref{eq:transfer_op_paired_bd}.
	
	Finally, let $ v\map{\bar\Delta}{\R} $ be $ d_\theta $-Lipschitz and let $ k\ge 1 $. Without loss of generality take $ \int v \, d\bar\mu_\Delta=0 $ and set $ w=L^{k-[k/2]}v $. By \cite[Lemma~2.2]{dedecker_prieur_intermittent}, $ \Lip(w)\le C\norm{v}_\theta$. By \eqref{eq:sup_transfer_op_base_bd}, it follows that
	\begin{align*}
		\sup_{\bar\Delta_0}|L^k v|=\sup_{\bar\Delta_0}|L^{[k/2]} w|&\le C(\Lip(w)|\theta^{\psi_{[k/2]}}|_1+|w|_1)\\
		&\le C(\norm{v}_\theta |\theta^{\psi_{[k/2]}}|_1+|L^{k-[k/2]}v|_1 ).
	\end{align*}
	Now by~\cite[Theorem~2.7]{korepanov2019explicit}, $ \pnorm{L^{k-[k/2]}v}_1\le C(k-[k/2])^{-(\beta-1)}\norm{v}_\theta $. By \cite[Lemma~5.5]{korepanov2019explicit},
	\begin{align*}
		\pnorm{\theta^{\psi_k}}_1 \le \frac{2}{\int \phi d\bar\mu_Y}\sum_{j>k/3}\bar\mu_Y(\phi>j)+k \sum_{\ell= 0}^\infty\theta^{\ell+1}\bar\mu_Y\biggl(\sum_{j=0}^{\ell-1}\phi \circ \bar F^j \ge \tfrac{1}{3}k\biggr).
	\end{align*} 
Thus $ |\theta^{\psi_{[k/2]}}|_1\le C[k/2]^{-(\beta-1)} $. It follows that
	$ \sup_{\bar\Delta_0}|L^k v|\le Ck^{-(\beta-1)}\norm{v}_\theta $,
	as required.
\end{proof}
	\subsection{Uniform families of nonuniformly hyperbolic maps}\label{subsection:uniform_families}
Let $ \tf_n\map{\ms}{\ms},\ n\ge 1 $ be a family of mixing nonuniformly hyperbolic maps as defined in Subsection~\ref{subsection:nuh_map_def} with invariant measures $ \mu_n $.
\begin{defn}\label{def:nuh_family}
	Let $ \beta>1 $. We call $ \tf_n\map{\ms}{\ms} $ a \textit{uniform family} of nonuniformly hyperbolic maps with $ O(k^{-\beta}) $ tails if:
	\begin{enumerate}
		\item The constants $ D_0>0$, $\theta\in(0,1) $ can be chosen independent of $ n\ge 1 $.
		\item There exists $ C_\phi>0 $ such that the return time functions $ \phi_n\map{Y_n}{\Z^+} $ satisfy $ \mu_{Y_n}(\phi_n\ge k)\le C_\phi k^{-\beta} $ for all $ n,k\ge 1 $.
		\item There exist $ \delta>0 $ and $ K>0 $ such that for all $ n $ there exists $ I_n\subset [1,K] $ with $ \mu_{Y_n}(\phi_n=k)\ge \delta $ for $ k\in I_n $ and $ \gcd\{I_n\}=\gcd\{\phi_n(y)\colon y \in Y_n\} $.
	\end{enumerate}
\end{defn}
We are now ready to state the main result of this section:
\begin{theorem}\label{thm:uniform_fcb}
	Let $ \tf_n $ be a uniform family of nonuniformly hyperbolic maps with  $ O(k^{-\beta}) $ tails. Then the family $ \tf_n $ satisfies the \fcb{} uniformly with rate $ k^{-(\beta-1)} $.
\end{theorem}
\begin{proof}
	Without loss of generality take $ \eta=1 $. Indeed, let $ d_\eta(x,y)=d(x,y)^\eta $. If $ \tf_n $ is a uniform family on $ (\ms,d) $, then $ \tf_n $ is a uniform family on $ (\ms,d_\eta) $ with slightly different constants, and separately $ \eta $-\holder{} functions with respect to $ d $ are separately Lipschitz with respect to $ d_\eta $.
	
By Lemma~\ref{lemma:fcb_explicit}, there exists a constant $ C>0 $ such that $ T=T_n $ satisfies \eqref{eq:dyn_holder_fcb} for all $ n\ge 1 $. Let $ G:\ms^q\to \R$, $ q\ge 1 $ be separately Lipschitz. Then by Proposition~\ref{prop:holder_implies_dyn_holder}, we have $ [G]_{\mathcal{H},i}\le D_0 [G]_{1,i} $ for all $ 0\le i<q $. Hence, for all $ n\ge 1 $, $ T_n $ satisfies the \fcb{} with constant $ C\max\{D_0,1\} $.
\end{proof}
\section{Examples}\label{section:examples}
In this section we consider examples of families of dynamical systems for which we can verify the hypotheses of Theorem~\ref{thm:homog_from_fcb_family}. 
\subsection{Intermittent Baker's maps}\label{subsection:intermittent_baker}
Let $ I=[0,1],$ $M=I^2 $. Fix a family of intermittent Baker's maps $ T_n:M\to M $, $ n\in \N\cup\{\infty\} $, as in~\eqref{eq:intermittent_baker} with parameters $ \alpha_n\in(0,\frac{1}{2}) $ such that $ \lim_{n\to\infty}\alpha_n = \alpha_\infty \in (0,\frac{1}{2})$. Recall that $ T_n $ is a skew product map of the form 
\[ T_n(x,z)=(\bar T_n(x),h_n(x,z)),\quad h_n(x,z)=\begin{cases}
	g_{n,0}(z), & x<\frac{1}{2},\\
	g_{n,1}(z), & x\ge\frac{1}{2},
\end{cases} \]
where $ \bar T_n:I\to I $ is the Liverani-Saussol-Vaienti map with parameter $ \alpha_n $ and $ g_{n,0}$, $g_{n,1} $ are the inverse branches of $\bar T_n $. The projection $ \pi:M\to I $, $ \pi(x,z)=x $ defines a semiconjugacy between $ T_n $ and $ \bar T_n $. By~\cite{lsvmaps}, there is a unique $\bar T_n $-invariant ergodic probability measure $ \bar\mu_n $ which is absolutely continuous with respect to Lebesgue. Let $ \mu_n $ be the unique $ T_n $-invariant ergodic probability measure such that $ \pi_*\mu_n =\bar\mu_n $ (we construct this measure in Proposition~\ref{prop:lifted_meas_skew}).

\begin{lemma}
	$ T_n $, $ n\in \N\cup\{\infty\} $, is a uniform family of nonuniformly hyperbolic maps with $ O(k^{-1/\alpha}) $ tails, where $ \alpha=\sup_n \alpha_n  $.
\end{lemma} 
\begin{proof}
For each $ n $, we take $ \bar Y=[1/2,1]$ and let $ \bar{\phi}_n:\bar Y\to \Z^+ $ be the first return time to $\bar Y $, i.e.\ $\bar \phi_n(y)=\inf\{k\ge 1:\bar T_n^k(y)\in \bar Y \} $. Then by~\cite[Section 6]{young1999recurrence}, the first return map $ \bar F_n=\bar T_n^{\bar \phi_n}:\bar Y\to \bar Y $ is a Gibbs-Markov map and there exists a constant $ C>0 $ such that
$\bar\mu_n(\bar\phi_n>k)\le Ck^{-1/\alpha_n}$ for all $ k\ge 1$.
 Moreover, by~\cite[Example 5.1]{kkmaveraging} both the constant $ C $ and the constants that appear in the definition of Gibbs-Markov map can be chosen independently of $ n $. It follows that condition (ii) in Definition~\ref{def:nuh_family} is satisfied.
 
 Note that $ \{\bar \phi_n = 1\}=[3/4,1] $. By~\cite[Lemma 2.4]{lsvmaps}, $ \inf_n \inf_I d\bar \mu_n/d\leb>0 $. Hence $ \inf_n \bar\mu_n|_{\bar Y}(\bar \phi_n = 1)>0 $, and so condition (iii) in Definition~\ref{def:nuh_family} is satisfied.

Let $ Y=\bar Y\times I $ and let $ \phi_n:Y\to \Z^+ $ be the first return time to $ Y $.  Then $ \phi_n=\bar\phi_n\circ \pi|_Y $ and $ \pi|_Y $ defines a semiconjugacy between $ F_n = T_n^{\phi_n}:Y\to Y $ and $ \bar F_n $.

We now complete the proof of condition (i) in Definition~\ref{def:nuh_family} by verifying that~\eqref{eq:two_sided_GM_contraction} and~\eqref{eq:yt_intermediate_iterate_bd} hold with constants $ D_0, \theta $ that are uniform in $ n $. Denote $ \psi_{n,k}(x)=\#\{j=0,\dots,k-1:\bar T_n^j x\in \bar Y\} $. We claim that 
\begin{equation}\label{eq:intermittent_baker_yt_bound}
	d(T_n^k(x_1,z_1),T_n^k(x_2,z_2))\le 2(\tfrac{1}{2})^{s(x_1,x_2)-\psi_{n,k}(x_1)}+(\tfrac{1}{2})^{\psi_{n,k}(x_1)}d(z_1,z_2).
\end{equation}
for all $ k\ge 1 $, $ n\in \N\cup\{\infty\} $, $ (x_1,z_1), (x_2,z_2)\in Y $. It is straightforward to check that both~\eqref{eq:two_sided_GM_contraction} and \eqref{eq:yt_intermediate_iterate_bd} follow from this claim with $ D_0=2 $, $ \theta=\frac{1}{2} $.

Let $ x_1, z_1, z_2\in I $. Then for $ i=1,2 $,
\begin{equation}\label{eq:skew_product_iterates}
	T^k_n(x_1,z_i)=(\bar T_n^k x_1, g_{n,a(k-1)}\circ \dots \circ g_{n,a(0)}(z_i)),
\end{equation}
where $ a(j)=\indic{\bar T_n^j x_1\in \bar Y} $. Since $ \norminf{g'_{n,0}}\ge 1 $ and $ \norminf{g'_{n,1}}=\frac{1}{2} $, by the mean value theorem it follows that 
\begin{align}\label{eq:intermittent_baker_stable_bd}
	d(T_n^k(x_1,z_1),T_n^k(x_1,z_2))\le \prod_{j=0}^{k-1}\big|g'_{n,a(j)}\big|_\infty d(z_1,z_2)\le (\tfrac{1}{2})^{\psi_{n,k}(x_1)}d(z_1,z_2).
\end{align}
Let $ x_1, x_2 \in \bar{Y} $. Without loss of generality assume that $ s(x_1,x_2)>\psi_{n,k}(x_1) $, for otherwise~\eqref{eq:intermittent_baker_yt_bound} is satisfied trivially. Since $ \phi_n $ is the first return time to $ \bar Y $, it follows that for all $ 0\le j<k $ we have $ \bar T_n^j x_1\in \bar Y $ if and only if $ \bar T_n^j x_2 \in \bar Y $. Hence by~\eqref{eq:skew_product_iterates},
$ d(T_n^k(x_1,z_1),T_n^k(x_2,z_1))=d(\bar T_n^k x_1,\bar T_n^k x_2). $ Note that $ \bar F_n^{\psi_{n,k}+1}=\bar T_n^r $, where $ r = \sum_{\ell=0}^{\psi_{n,k}}\bar\phi_n \circ \bar F_n^\ell $. Since $ r(x_1)=r(x_2)>k $ and $ \bar T'_n\ge 1 $, it follows that 
\[ d(\bar T_n^k x_1,\bar T_n^k x_2)\le d(\bar T_n^{r(x_1)} x_1,\bar T_n^{r(x_1)} x_2)=d(\bar F_n^{\psi_{n,k}+1} x_1,\bar F_n^{\psi_{n,k}+1} x_2). \]
Now $\bar T_n'\ge 1 $ and $\bar T_n'=2 $ on $ \bar Y $ so $ d(\bar F_n y, \bar F_n y')\ge 2d(y,y') $ whenever $ y,y'\in \bar Y $ belong to the same partition element. By~\cite[Lemma 3.2]{alves_book}, it follows that $ d(\bar F_n^{\psi_{n,k}+1} x_1,\bar F_n^{\psi_{n,k}+1} x_2)\le (\frac{1}{2})^{s(x_1,x_2)-(\psi_{n,k}(x_1)+1)} $. The claim~\eqref{eq:intermittent_baker_yt_bound} follows by combining this with~\eqref{eq:intermittent_baker_stable_bd}. 
\end{proof}
By Theorem~\ref{thm:uniform_fcb}, it follows that the family $ T_n $ satisfies the Functional Correlation Bound uniformly with rate $ k^{-(1/\alpha - 1)} $. In the remainder of this subsection we verify that condition~\eqref{eq:conv_correlation_fns} is satisfied by verifying the conditions (A1), (A2) from Remark~\ref{rk:sufficient_conditions_conv}. It follows that our main result, Theorem~\ref{thm:homog_from_fcb_family}, applies to the family $ T_n $.

By~\cite{korepanovlinearresponse,baladitodd}, $ \bar\mu_n $ is \textit{strongly statistically stable}, i.e.\ $ d\bar \mu_n/d\leb\to d\bar\mu_\infty/d\leb $ in $ L^1(\leb) $.  The following proposition immediately implies that (A2) holds and will also be useful in the proof that $ \mu_n$ is statistically stable, as required by (A1).
\begin{prop}\label{prop:T_n_conv_prob}
	For all $ a>0 $, for all $ j\in \N $,
	\begin{equation}\label{eq:mu_bar_n_conv_in_prob}
		\lim_{n\to\infty}\bar\mu_n\Big(x:\sup_{z\in I}d(T_n^j(x,z),T_\infty^j(x,z))>a \Big)=0.
	\end{equation}
\end{prop}
\begin{proof}
	Let $ \eps>0 $. Choose $ K\subset I $ compact such that $ \frac{1}{2}\notin \bar T_\infty^i(K) $ for $ 0\le i\le j $ and $ \bar \mu_\infty(K)\ge 1-\eps $. Then $ T_n^j \to T_\infty^j $ uniformly on $ K\times I $ so for all $ a>0 $,
	\[ \limsup_{n\to\infty}\bar\mu_\infty\Big(x:\sup_{z\in I}d(T_n^j(x,z),T_\infty^j(x,z))>a \Big)\le \bar\mu_\infty(I\setminus K)<\eps. \]
	It follows that~\eqref{eq:mu_bar_n_conv_in_prob} holds with $ \bar\mu_\infty $ in place of $ \bar\mu_n $. The inequality~\eqref{eq:mu_bar_n_conv_in_prob} follows by strong statistical stability.
\end{proof}
Note that $ T_\infty $ is continuous on $ I^2\setminus (\{\frac{1}{2}\}\times I) $ so $ T_\infty $ is continuous $ \mu_\infty $-a.e. In the remainder of this subsection, we complete the verification of condition (A2) by showing that $ \mu_n $ is statistically stable. We closely follow the strategy that Alves \& Soufi~\cite{alvessoufi} used to prove statistical stability for the Poincar\'e maps of geometric Lorenz attractors.

First let us recall the standard procedure for constructing invariant measures for skew products with contracting fibres. Given a bounded, measurable function $ \phi:M\to \R $, define $ \phi^+:I\to \R $ by 
$ \phi^+(x)=\sup_{z\in I}\phi(x,z). $
\begin{prop}\label{prop:lifted_meas_skew}
	Let $ n\in \N\cup \{\infty\} $. There exists a unique probability measure $ \mu_n $ such that for any continuous function $ \phi:M\to \R $,
	\begin{equation}\label{eq:lifted_meas_def}
		\int_M \phi\, d\mu_n = \lim_{m\to \infty}\int_{I} (\phi\circ T_n^m)^{+}d\bar\mu_n.
	\end{equation}
	Moreover, the convergence is uniform in $ n $. Besides, $ \mu_n $ is the unique $ T_n $-invariant ergodic probability measure such that $ \pi_* \mu_n = \bar\mu_n $.
\end{prop}
\begin{proof}
We first show that the maps $ T_n $ uniformly contract fibres in the sense that 
\begin{equation}\label{eq:unif_contraction}
	\mathrm{diam}\, T_n^m\pi^{-1}(x) \to 0 \text { as }m\to \infty, \text{ uniformly in $ x $ and $ n $.}
\end{equation}
Fix $ x $ and $ n $. By~\eqref{eq:skew_product_iterates} and the fact that $ g_{n,0} $ and $ g_{n,1} $ are inverse branches of $ \bar T_n $,
\[ T_n^m\pi^{-1}(x)=\{\bar T_\infty^m(x)\}\times H(I) \]
where $ H:I\to I $ is an inverse branch of $ \bar T^m_n $. By~\cite[equation (5)]{leppanen2017functional}, there exists $ C>0 $ such that for all $ m,n\ge 1 $, for any inverse branch $ H $ of $ \bar T_n^m $ we have
$ \mathrm{diam}\,H(I)\le Cm^{-1/\sup_n\alpha_n}. $
This proves~\eqref{eq:unif_contraction}. The rest of the proof that the limit~\eqref{eq:lifted_meas_def} exists and the convergence is uniform in $ n $ proceeds exactly as in~\cite[Proposition 3.3]{alvessoufi} (with $ P_n $ and $ f_n $ changed to $ T_n $ and $ \bar T_n $). In~\cite[Corollary 6.4]{singularhyperbolicchaotic} it is shown that $ \mu_n $ indeed defines a $ T_n $-invariant probability measure and that the ergodicity of $ \mu_n $ follows from the ergodicity of $ \bar\mu_n $. By~\cite[Remark 2(b)]{melbourne_butterley}, $ \mu_n $ is the unique $ T_n $-invariant ergodic probability measure such that $ \pi_* \mu_n =\bar\mu_n $.
\end{proof}
\begin{prop}\label{prop:swapped_lims}
	For all $ m\ge 1 $,
	\[ \lim_{n\to \infty}\int (\phi \circ T_n^m)^{+}d\bar \mu_n = \int (\phi \circ T_\infty^m)^{+}d\bar \mu_\infty. \]
\end{prop}
\begin{proof}We proceed as in~\cite[Lemma 3.2]{alvessoufi}.
	
	Write 
	$ \int (\phi \circ T_n^m)^{+}d\bar \mu_n - \int (\phi \circ T_\infty^m)^{+}d\bar \mu_\infty=I_n+J_n $, 
	where 
	\begin{align*}
		I_n&= \int (\phi \circ T_n^m)^{+}-(\phi \circ T_\infty^m)^{+}\, d\bar \mu_n,\quad
		J_n= \int (\phi \circ T_\infty^m)^{+}d\bar \mu_n - \int (\phi \circ T_\infty^m)^{+}d\bar \mu_\infty.
	\end{align*}
Now for all $ x\in I $, 
\begin{equation}\label{eq:diff_phi_+_ineq}
	|(\phi \circ T_n^m)^{+}(x)-(\phi \circ T_\infty^m)^{+}(x)|\le \sup_{y\in I}|\phi \circ T_n^m(x,z)-\phi \circ T_\infty^m(x,z)|.
\end{equation}
Since $ M $ is compact, $ \phi $ is uniformly continuous on $ M $. Hence for any $ \eps>0 $ there exists $ \delta>0 $ such that $ |\phi(z)-\phi(z')|<\eps $ for all $ z,z'\in M $ with $ d(z,z')<\delta $. Let 
\[ S = \Big\{x\in I:\sup_{y\in I}d(T_n^m(x,y),T_\infty^m(x,y))\ge\delta\Big\}. \]
Then by~\eqref{eq:diff_phi_+_ineq},
\begin{align*}
	|I_n|&\le \int_{S} |(\phi \circ T_n^m)^{+}-(\phi \circ T_\infty^m)^{+}|\, d\bar \mu_n + \int_{I\setminus S} |(\phi \circ T_n^m)^{+}-(\phi \circ T_\infty^m)^{+}|\, d\bar \mu_n\\
	&\le 2\norminf{\phi}\bar\mu_n(S)+\eps.
\end{align*}
By Proposition~\ref{prop:T_n_conv_prob}, $ \bar \mu_n(S)\to 0 $ as $ n\to \infty $. Since $ \eps>0 $ is arbitrary, it follows that $ I_n\to 0 $.

Finally, note that
\begin{align*}
	|J_n|= \bigg|\int (\phi \circ T_\infty^m)^{+}\bigg(\frac{d\bar\mu_n}{d\leb} - \frac{d\bar\mu_\infty}{d\leb}\bigg) d\leb\bigg|\le \norminf{\phi}\bigg|\frac{d\bar\mu_n}{d\leb} - \frac{d\bar\mu_\infty}{d\leb}\bigg|_{L^1(\leb)}.
\end{align*}
Hence by strong statistical stability, $ J_n\to 0 $ as $ n\to \infty $.
\end{proof}
We can now complete the proof that $ \mu_n $ is statistically stable, i.e.\ $ \mu_n\wconv \mu_\infty $. 
	Let $ \phi:M\to \R $ be continuous. Then by Propositions~\ref{prop:lifted_meas_skew} and \ref{prop:swapped_lims},
	\[ \int_M \phi\, d\mu_\infty = \lim_{m\to \infty}\int_{I}(\phi\circ \bar T_\infty^m)^{+}d\bar\mu_\infty=\lim_{m\to \infty}\lim_{n\to\infty}\int_{I}(\phi\circ \bar T_n^m)^{+}d\bar\mu_n. \]
	Since $ \int_{I} (\phi\circ \bar T_n^m)d\bar\mu_n\to \int_M \phi\, d\mu_n$ as $ m\to\infty $ uniformly in $ n $, we can swap the limits as $ m\to \infty $ and $ n\to \infty $ in the above expression. Thus
	\[ \int_M \phi\,d\mu_\infty = \lim_{n\to \infty}\lim_{m\to\infty}\int_{I}(\phi\circ \bar T_n^m)^{+}d\bar\mu_n=\lim_{n\to \infty}\int_M \phi\, d\mu_n, \]
	as required.
\subsection{Externally forced dispersing billiards}\label{subsection:sinai_billiard_external}
A Sinai billiard table on the two-torus $ \mathbb{T}^2 $ is a set of the form $ Q=\mathbb{T}^2\setminus \cup_i B_i $ where $ \{B_i\} $ is a finite collection of open sets such that $\bar B_i \cap \bar B_j =\emptyset $ for $ i\ne j $. It is assumed that the sets $ B_i $ have $ C^3 $ boundaries with positive curvature. The billiard flow on $ Q\times S^1 $ is induced by the motion of a particle that moves in straight lines at unit speed on $ Q $ and collides elastically with the boundary $ \partial Q $. We say that the table has finite horizon if there exists a constant $ L>0 $ such that any line of length $ L $ in $ \mathbb{T}^2 $ intersects $ \partial Q $.

In~\cite{chernov_small_external_I,chernov_small_external_II} Chernov studied perturbations of the finite horizon Sinai billiard flow where a small stationary force $ F $ acts on the particle between its collisions with $ \partial Q $. We refer to~\cite[Section 2]{chernov_small_external_I} for the precise details of the model. In particular, it is assumed that the force preserves an additional integral of motion and that the phase space obtained by restricting to one of its level sets is a compact 3-dimensional manifold. 

Consider the flow obtained by restricting to one of these level sets. The assumptions then guarantee that the collision map $ T_F $ with the table can be parametrised on the same space $ M = \partial Q\times [-\pi/2,\pi/2] $ as the collision map of the unperturbed Sinai billiard flow. Let $ (F_n)_{n\in \N} $ be a sequence of admissible forces such that $ F_n\to F_{\infty}=0 $ in $ C^2 $ and define $ T_n=T_{F_n}:M\to M $. Let $ \mu_n $ denote the unique SRB measure for $ T_n $. 
\begin{prop}\label{prop:fcb_externally_forced}
	For all $ \gamma>1 $ the family $ T_n:M\to M $, $ n\in \N\cup\{\infty\} $, satisfies the Functional Correlation Bound uniformly with rate $ k^{-\gamma} $.
\end{prop}
\begin{remark}
	In principle, it should be possible to prove this proposition by verifying that the family $ T_n $ is a uniform family of nonuniformly hyperbolic maps and applying Theorem~\ref{thm:uniform_fcb}. Indeed, for each $ n $, the system $ T_n $ is modelled by a Young tower with exponential tails~\cite{chernov_small_external_I}. However, the construction of the base of the tower in~\cite{chernov_small_external_I} is quite intricate so it seems difficult to check condition (iii) in Definition~\ref{def:nuh_family}. 
\end{remark}
\begin{proof}
	In \cite{leppanen2017billiards}, Lepp\"{a}nen \& Stenlund considered the finite horizon Sinai billiard map and proved a functional correlation bound for separately dynamically H\"{o}lder functions. (Note that their definition of dynamical H\"{o}lder continuity differs from that in Section~\ref{subsection:explicit_fcb}.) Recall the definition of the past/future separation times $ s_{\pm} $ and the dynamically H\"{o}lder function classes $ \mathcal{H}_{\pm} $ from~\cite[Section 2]{leppanen2017billiards}.
	
	We first show that separately H\"{o}lder functions are separately dynamically H\"{o}lder with parameters independent of $ n $. By \cite[p.95]{chernov_small_external_II}, there exist constants $ C>0 $, $ \Lambda>1 $ independent of $ n $ such that $ d(x,y)\le C\Lambda^{-s_{+}(x,y)} $ whenever $ x,y\in M $ belong to the same local unstable manifold. Similarly, $ d(x,y)\le C\Lambda^{-s_{-}(x,y)} $ whenever $ x,y\in M $ belong to the same local stable manifold. Let $ v\in C^\eta(M) $. It follows that 
	\[ |v(x)-v(y)|\le [v]_{\eta}d(x,y)^{\eta}\le C^\eta[v]_{\eta}(\Lambda^{-\eta})^{s_{+}(x,y)} \]
	whenever $ x $ and $ y $ belong to the same local unstable manifold. Hence $ v\in \mathcal{H}_{+}(C^\eta[v]_{\eta},\Lambda^{-\eta}) $. Similarly, $ v\in \mathcal{H}_{-}(C^\eta[v]_\eta,\Lambda^{-\eta}) $. Let $ G:M^q\to\R $, $ q\ge 1 $ be separately $ \eta $-H\"{o}lder and set $ c=C^\eta \max_i [G]_{\eta,i} $, $ \vartheta =\Lambda^{-\eta} $. Then $ G(x_0,\dots,x_{i-1},\cdot,x_{i+1},\dots,x_{q-1})\in \mathcal{H}_{-}(c,\vartheta)\cap \mathcal{H}_{+}(c,\vartheta) $ for all $ x_0,\dots,x_{q-1}\in M $, $ 0\le i<q $.
	
	It remains to explain why the arguments used in the proof of~\cite[Theorem 2.4]{leppanen2017billiards} go through with system constants $ M_0, M_1 $ and $ \theta_0$,  $\theta_1$ uniform in $ n $. The result then follows by applying~\cite[Theorem 2.4]{leppanen2017billiards} with $ K=2 $, $ F=G $ and $ c,\vartheta $ as defined above.
	
	Note that~\cite[Lemma 4.1]{leppanen2017billiards} merely gives the usual decomposition of $ \mu $ into a standard family. Let $ \{(\xi_q,\nu_q):q\in \mathcal{Q}\} $ be as defined in that lemma. By~\cite[p.\ 96]{chernov_small_external_II}, $ \{(\xi_q,\nu_q):q\in\mathcal{Q}\} $ is a proper standard family. In particular, there exists a constant $ M_1>0 $ independent of $ n $ such that $ \lambda(\{q\in \mathcal{Q}:|\xi_q|\le \eps\})\le M_1\eps $ for all $ \eps>0 $. By~\cite[p.171]{chernov_markarian}, it follows that there exists a constant $ M'_1 $ independent of $ n $ such that $ \int_{\mathcal{Q}}|\xi_q|^{-1}d\lambda(q)\le M'_1 $, so~\cite[Lemma 4.3]{leppanen2017functional} goes through. Finally, it follows from the growth lemma (\cite[Proposition 5.3]{chernov_small_external_I}) and the equidistribution property (\cite[Proposition 2.2]{chernov_small_external_II}) that~\cite[Lemma 4.2]{leppanen2017billiards} goes through with constants $ a_0 $, $ M_0 $ and $ \theta_0 $ that are uniform in $ n $. The rest of the proof of~\cite[Theorem 2.4]{leppanen2017billiards} proceeds exactly as in~\cite{leppanen2017billiards}.
\end{proof}
We finish this subsection by showing that condition~\eqref{eq:conv_correlation_fns} is satisfied. It follows that Theorem~\ref{thm:homog_from_fcb_family} applies to the family $ T_n $. 
\begin{prop}\label{prop:conv_correlations_forced}
Condition~\eqref{eq:conv_correlation_fns} is satisfied.
\end{prop}
\begin{proof}
	For $ n\in\N\cup\{\infty\} $, Demers \& Zhang~\cite{demers_zhang_perturbations} considered the action of the transfer operator $ \mathcal{L}_n $ associated with $ T_n $ on certain spaces of distributions. In particular, if $ \nu $ is a finite signed measure, then $ \mathcal{L}_n \nu=(T_n)_{*}\nu $.
		
		 Fix $ \eta\in (0,1] $. The article~\cite{demers_zhang_perturbations} constructs Banach spaces $ (\mathcal{B},\norm{\cdot}_{\mathcal{B}}) $ and $ (\mathcal{B}_w,\norm{\cdot}_{\mathcal{B}_w}) $ with the following properties: 
		\begin{enumerate}
			\item There is a sequence of continuous embeddings $ \mathcal{B}\hookrightarrow \mathcal{B}_w \hookrightarrow (C^\eta(M))'$.
			\item For each $ n $, $ \mathcal{L}_n $ is a well-defined bounded linear operator on both $ \mathcal{B} $ and $ \mathcal{B}_w $. Moreover, $ \sup_n \norm{\mathcal L_n}_{\mathcal{B}_w}<\infty $.
			\item (\cite[Theorem 2.2]{demers_zhang_perturbations}) For each $ n $, we have $ \mu_n \in \mathcal{B} $ and $ \mu_n $ is the unique element of $ \mathcal{B} $ such that $ \mathcal{L}_n \mu_n = \mu_n $ and $ \mu_n(1)=1 $.
			\item (\cite[Theorem 2.11]{demers_zhang_perturbations}) $ \norm{\mathcal{L}_n - \mathcal{L}_\infty}_{\mathcal{B}\to \mathcal{B}_w}\to 0 $ as $ n\to\infty $. By~\cite[Theorem 2.1]{demers_zhang_perturbations}, it follows that $ \mu_n \to \mu_\infty $ in $ \mathcal{B}_w $.
			\item (\cite[Lemma 5.3]{demers_zhang_perturbations}) Let $ v\in C^\eta(M) $. Then $ v $ is a bounded multiplier on $ \mathcal{B} $ (that is, $ h\mapsto vh $ is a well-defined bounded operator on $ \mathcal{B} $). Moreover, $ v $ is a bounded multiplier on $ \mathcal{B}_w $.\footnote{Since the definitions of the weak norm and the strong stable norm are similar, this follows easily from the arguments used to bound the strong stable norm at the beginning of the proof of~\cite[Lemma~5.3]{demers_zhang_perturbations}. }
		\end{enumerate}
	Fix $ v\in C^\eta(M) $ and $ k\ge 0 $. For $ n\in\N\cup\{\infty\} $, define a signed probability measure $ \nu_n $ by $ \nu_n=v\mu_n $. Then $ \mathcal{L}_n^k \nu_n(w)=\nu_n(w\circ T_n^k)=\mu_n(vw\circ T_n^k) $ for all $ w\in C^\eta(M) $, so it sufficient to prove that $ \mathcal{L}_n^k \nu_n \to\mathcal{L}_\infty^k \nu_\infty $ in $ (C^\eta(M))' $. Now $ \mu_n\in\mathcal{B} $ so by properties (v) and (ii) we have $ \nu_n\in \mathcal{B} $ and $ \mathcal{L}_n^k \nu_n\in \mathcal{B} $. Hence by (i), it suffices to show that $ \mathcal{L}_n^k \nu_n\to \mathcal{L}_\infty^k \nu_\infty $ in $ \mathcal{B}_w $. 
	
	Write $\mathcal{L}_n^k \nu_n -\mathcal{L}_\infty^k \nu_\infty=I_n+J_n $, where $ I_n = \mathcal{L}^k_n (\nu_n - \nu_\infty) $ and 
\begin{align*}
	J_n = \mathcal{L}^k_n\nu_\infty-\mathcal{L}^k_\infty\nu_\infty= \sum_{j=0}^{k-1}\mathcal{L}_n^j(\mathcal{L}_n - \mathcal{L}_\infty)\mathcal{L}_\infty^{k-j-1} \nu_\infty.
\end{align*}
By properties (ii) and (v),
\begin{align*}
	\norm{I_n}_{\mathcal{B}_w}\le \norm{\mathcal{L}^k_n}_{\mathcal{B}_w}\norm{v\mu_n - v\mu_\infty}_{\mathcal{B}_w}\le C\norm{\mu_n - \mu_\infty}_{\mathcal{B}_w},
\end{align*}
where $ C $ depends only on $ v $. By (iv), it follows that $ \lim_{n\to \infty}\norm{I_n}_{\mathcal{B}_w}=0$. Now by property (ii), there exists a constant $ C'>0 $ such that
\begin{align*}
	\norm{J_n}_{\mathcal{B}_w}&\le \sum_{j=0}^{k-1} \norm{\mathcal{L}^j_n}_{\mathcal{B}_w}\norm{\mathcal{L}_n - \mathcal{L}_\infty}_{\mathcal{B}\to \mathcal{B}_w}\norm{\mathcal{L}_\infty^{k-j-1}}_{\mathcal{B}} \norm{\nu_\infty}_{\mathcal{B}}\\
	&\le C'\norm{\mathcal{L}_n - \mathcal{L}_\infty}_{\mathcal{B}\to \mathcal{B}_w}\!\norm{\nu_\infty}_{\mathcal{B}}.
\end{align*}
Hence by (iv), $ \lim_{n\to \infty}\norm{J_n}_{\mathcal{B}_w}=0 $, which completes the proof.
\end{proof}
\begin{remark}
	Note that we have not used any facts about the anisotropic Banach spaces in~\cite{demers_zhang_perturbations} apart from properties (i)-(v). These properties arise naturally in situations where statistical stability is proved by Keller-Liverani perturbation theory (see e.g.~\cite{gouezel_liverani_anisotropic,demers_liverani_pw_hyperbolic}). 
\end{remark}
%\subsection{Further examples}
%	Consider the family of H\'{e}non maps $ \tf_{a,b}\map{\R^2}{\R^2} $ given by $ \tf_{a,b}(x,y)=(1-ax^2+y,bx) $ where $ a\in (1,2), b>0 $. Benedicks \& Carleson \cite{benedickscarleson91} showed that there is a positive Lebesgue measure set $ \mathcal{BC} $ such that for $ (a,b)\in \mathcal{BC} $ the map $ \tf $ has a non-hyperbolic attractor. Benedicks \& Young~\cite{benedicksyoung93,henon_young_tower} showed that for $ (a,b)\in \mathcal{BC} $ there is a unique SRB measure supported on the attractor and $ T_{a,b} $ is modelled by a Young tower with exponential tails. For $ n\in \N\cup \{\infty\} $ let $ (a_n,b_n)\in \mathcal{BC} $ with $ \lim_{n\rightarrow \infty}(a_n,b_n)=(a_\infty,b_\infty) $ and define $ \tf_n=\tf_{a_n,b_n}. $ By \cite{acfhenonss}, the family $ \tf_n $ is statistically stable so (A1) holds. Since $ T^j_n \to T^j_\infty $ uniformly for all $ j\in \N $, (A2) holds. \tcr{Language on saying effectively checked uniform family but without exact definitions}
%\begin{itemize}
%	\item Give example of Henon maps from \cite{afcyoungtowerss}. See other document for notes on how to show that condition (iii) in Definition~\ref{def:nuh_family} holds
%	\item Give example of Lorenz maps. In \cite{larkin_random_lorenz} it is checked that such maps are modelled by Young towers with uniform constants, including condition (iii) in Definition~\ref{def:nuh_family}
%	
%\end{itemize}
%\tcr{Comment about how we can cover examples from~\cite{korepanov2022deterministic} i.e.\ Viana maps and unimodal maps}

\appendix
\section{A limit theorem for triangular arrays of random vectors}
In this appendix, we state and prove an iterated WIP for triangular arrays of random vectors. Our assumptions are similar to those of Lyapunov's classical central limit theorem.
\begin{lemma}\label{lemma:lindeberg_wip}
	Fix $d\ge 1$. Let $(\chi_{n,i})_{n\ge 1,1\le i\le k_n}$ be an array of mean zero $\R^d$-valued random vectors such that $(\chi_{n,i})_{1\le i\le k_n}$ are independent for each $n\ge 1$. Suppose that 
	\begin{enumerate}[label=(\roman*)]
		\item There exists a matrix $\Sigma\in \R^{d\times d}$ such that for all $ t\in [0,1], $
		\[\lim_{n\rightarrow \infty}\Ex{\sum_{i=1}^{[tk_n]}\chi_{n,i}\otimes \sum_{i=1}^{[tk_n]}\chi_{n,i}}=t\Sigma.\]
		\item There exists $p>1$ such that $\lim_n \sum_{i=1}^{k_n}\Ex{|\chi_{n,i}|^{2p}}=0.$
	\end{enumerate}
	Define \cadlag{} processes $\shat{W}_n\in D([0,1],\R^d) ,\shat{\bbW}_n\in D([0,1],\R^{d\times d})$ by
	\begin{equation*}
		\shat W_n(t)=\sum_{1\le i\le[tk_n]}\chi_{n,i}\quad \text{and}\quad\shat{\bbW}_n(t)=\sum_{1\le i<j\le  [tk_n]}\chi_{n,i}\otimes \chi_{n,j}.
	\end{equation*}
	Then $(\shat W_n,\shat\bbW_n)\wconv (W,\int W\otimes dW)$ in $D([0,1],\R^d \times \R^{d\times d})$ with sup-norm topology, where $W$ is a Brownian motion with covariance $\Sigma$.
\end{lemma}
\begin{proof}
	For $ n\ge 1 $, $ t\in [0,1] $ let $ \mathcal{F}_t^n $ be the $ \sigma $-algebra generated by $ \chi_{n,1},\ldots,\chi_{n,[tk_n]} $. Then $ (\mathcal{F}^n_t)_{t\in[0,1]} $ forms a filtration and $ \shat{W}_n $ is an $  (\mathcal{F}^n_t)_{t\in[0,1]} $-martingale.
	
	We first show that $ \shat W_n \wconv W $ by verifying the hypotheses of \cite[Theorem~2.1]{whittmartingalefclt}. 	Let $ \varepsilon>0 $. Note that
	\begin{align}\label{eq:lindeberg_condition}
		\sum_{i=1}^{k_n}\Ex{|\chi_{n,i}|^2\indic{|\chi_{n,i}|>\varepsilon}}\le \sum_{i=1}^{k_n}\frac{1}{\varepsilon^{2p-2}}\Ex{|\chi_{n,i}|^{2p}\indic{|\chi_{n,i}|>\varepsilon}}\longrightarrow 0
	\end{align}
	by Lemma~\ref{lemma:lindeberg_wip}(ii). Now for all $ n\ge 1 $, $ 1\le i\le k_n $ we have
	\begin{align*}
		|\chi_{n,i}|^2\le \varepsilon^2 +|\chi_{n,i}|^2\indic{|\chi_{n,i}|^2>\varepsilon}.
	\end{align*}
	Hence 
	\begin{align*}
		\bigg(\mathbb{E}\bigg[\max_{1\le i\le k_n} |\chi_{n,i}|\bigg]\bigg)^2\le\Ex{\max_{1\le i\le k_n} |\chi_{n,i}|^2}&\le \varepsilon^2 + \Ex{\max_{1\le i\le k_n} |\chi_{n,i}|^2\indic{|\chi_{n,i}|^2>\varepsilon}}\\
		&\le \varepsilon^2 + \Ex{\sum_{i=1}^{k_n}|\chi_{n,i}|^2\indic{|\chi_{n,i}|>\varepsilon}}.
	\end{align*}
	Hence by \eqref{eq:lindeberg_condition},
	$ \limsup_n \Ex{\max_{1\le i\le k_n} |\chi_{n,i}|}\le \varepsilon. $ Since $ \varepsilon>0 $ is arbitrary, it follows that $ \limsup_n \Ex{\max_{1\le i\le k_n} |\chi_{n,i}|}=0$.
	
	Fix $ t\in [0,1] $ and $ 1\le a,b\le d $. Let $ [\shat{W}^a_n,\shat W_n^b](t) $ denote the quadratic covariation of $ \shat W^a_n $ and $ \shat W^b_n $ at time $ t $. By condition (i) of Lemma~\ref{lemma:lindeberg_wip},
	\begin{align}\label{eq:limiting_ex_chi_dot_chi_sum}
		\lim_{n\rightarrow \infty}\Ex{[\shat{W}^a_n,\shat W_n^b](t)}=\lim_{n\rightarrow \infty}\mathbb{E}\Biggl[\sum_{i=1}^{[tk_n]}\chi^a_{n,i} \chi^b_{n,i}\Biggr]=t\Sigma_{ab}.
	\end{align}
	Let $ 1<p<2 $ be as in Lemma~\ref{lemma:lindeberg_wip}. By von Bahr-Esseen's inequality, 
	\begin{align*}
		\Ex{\Biggl|\sum_{i=1}^{[tk_n]}(\chi^a_{n,i}\, \chi^b_{n,i}-\Ex{\chi^a_{n,i}\, \chi^b_{n,i}}) \Biggr|^{p}}&\ll \sum_{i=1}^{[tk_n]}\Ex{\big|\chi^a_{n,i}\, \chi^b_{n,i}-\Ex{\chi^a_{n,i}\, \chi^b_{n,i}}\! \big|^{p}}.
	\end{align*}
	Now, 
	\begin{align*}
		\Ex{\big|\chi^a_{n,i}\, \chi^b_{n,i}-\Ex{\chi^a_{n,i}\, \chi^b_{n,i}}\! \big|^{p}}&\le 2^{p}\Ex{|\chi^a_{n,i}\, \chi^b_{n,i}|^{p}}\le 2^{p}\Ex{|\chi_{n,i}|^{2p}}.
	\end{align*}
	By condition (ii) of Lemma~\ref{lemma:lindeberg_wip}, it follows that $ \sum_{i=1}^{[tk_n]}(\chi^a_{n,i}\, \chi^b_{n,i}-\Ex{\chi^a_{n,i}\, \chi^b_{n,i}})\rightarrow 0 $ in $ L^p $. Hence by \eqref{eq:limiting_ex_chi_dot_chi_sum}, 
	\begin{align}\label{eq:quadratic_covar_conv_lp}
		[\shat{W}^a_n,\shat W_n^b](t)=\sum_{i=1}^{[tk_n]}\chi^a_{n,i}\, \chi^b_{n,i}\longrightarrow t\Sigma_{ab}\text{ in }L^p.
	\end{align}
	Hence we have verified the hypotheses of \cite[Theorem~2.1(i)]{whittmartingalefclt}, it follows that $ \shat W_n\wconv W $ in $ D([0,1],\R^d). $
	
	We complete the proof of this lemma by verifying the hypotheses of~\cite[Theorem 2.2]{kurtzprotter} (with $ \delta=\infty $ and $ A_n\equiv 0 $). By the continuous mapping theorem, $ (\shat W_n,\shat W_n)\wconv (W,W) $ in $ D([0,1],\R^d\times \R^d)$. By \eqref{eq:quadratic_covar_conv_lp}, $ \sup_n \Ex{[\shat W_n^a,\shat W_n^a](t)}<\infty $ for all $ t\in[0,1] $ and $ 1\le a\le d $, so condition C2.2(i) in~\cite[Theorem~2.2]{kurtzprotter} is satisfied. Hence, by \cite[Theorem~2.2]{kurtzprotter} it follows that $ (\shat W_n,\shat \bbW_n)\wconv {(W,\int W\otimes dW)} $ in $ D([0,1],\R^d\times \R^{d\times d}) $ with Skorohod topology. Since $ (W, \int W\otimes dW) $ has continuous sample paths, by \cite[Section 15]{billingsley_conv_prob} we also have weak convergence in the sup-norm topology.
\end{proof}
\bibliographystyle{alpha}
\bibliography{fcb_homogenisation_families}
\end{document}